\pdfoutput=1

\documentclass[letterpaper, 11pt]{amsart}
\usepackage{amsmath, amsfonts, amssymb, amsthm, amsrefs, array, stmaryrd, graphicx, hyperref, eucal, mathrsfs}
\usepackage[hhmmss]{datetime}

\newtheorem{thm}{Theorem}[section]
\newtheorem{lemma}{Lemma}[section]
\newtheorem{prop}[lemma]{Proposition}
\newtheorem{cor}[lemma]{Corollary}
\newtheorem{conj}[lemma]{Conjecture}
\theoremstyle{definition}

\newtheorem{remark}[lemma]{Remark}
\numberwithin{equation}{section}

\newcommand{\II}{\text{I\hspace{-0.050cm}I}}

\def\esmath{\ensuremath}

\def\phv{\ensuremath\varphi}

\def\NN{\esmath\mathbb N} 
\def\RR{\esmath\mathbb R} 

\newcommand{\p}{\partial}

\newcommand{\bn}{\begin{enumerate}}
\newcommand{\en}{\end{enumerate}}
\newcommand{\bi}{\begin{itemize}}
\newcommand{\ei}{\end{itemize}}
\newcommand{\bqq}{\begin{eqnarray*}}
\newcommand{\eqq}{\end{eqnarray*}}
\newcommand{\balg}{\begin{align*}}
\newcommand{\ealg}{\end{align*}}

\DeclareMathOperator{\rc}{Ric}
\DeclareMathOperator{\Rm}{Rm}

\begin{document}
\title[Type-$\II$ singularities in Ricci flow on $\RR^{N}$]{On type-$\II$ singularities in Ricci flow on $\RR^{N}$}
\author{Haotian Wu}
\address{Department of Mathematics, University of Oregon, Eugene, OR 97403, USA}
\email{hwu@uoregon.edu}
\date{\usdate\today}


\begin{abstract}
In each dimension $N\geq 3$ and for each real number $\lambda\geq 1$, we construct a family of complete rotationally symmetric solutions to Ricci flow on $\RR^{N}$ which encounter a global singularity at a finite time $T$. The singularity forms arbitrarily slowly with the curvature blowing up arbitrarily fast at the rate $(T-t)^{-(\lambda+1)}$. Near the origin, blow-ups of such a solution converge uniformly to the Bryant soliton. Near spatial infinity, blow-ups of such a solution converge uniformly to the shrinking cylinder soliton. As an application of this result, we prove that there exist standard solutions of Ricci flow on $\RR^N$ whose blow-ups near the origin converge uniformly to the Bryant soliton.
\end{abstract}

\maketitle

\section{Introduction}
An important phenomenon in Ricci flow is the formation of finite time singularities which occurs for a large family of initial metrics. Let $(M,g)$ be a complete Riemannian manifold and $g(t)$ be a solution to Ricci flow 
\begin{align*}
 \p_t g = -2\rc(g)
\end{align*}
for time $t\geq 0$. Suppose $g(t)$ becomes singular at time $T<\infty$, then this finite time singularity is called \emph{Type-I} if
\begin{align*}
 \sup\limits_{M\times[0,T)} \left\vert \Rm(\cdot, t) \right\vert (T-t) < \infty,
\end{align*}
and it is called \emph{Type-$\II$} if
\begin{align*}
 \sup\limits_{M\times[0,T)} \left\vert \Rm(\cdot, t) \right\vert (T-t) = \infty.
\end{align*}

The simplest example of a Type-I singularity in Ricci flow is the shrinking round sphere. In his seminal paper \cite{H82}, Hamilton proved that Ricci flow on a compact three-manifold with positive Ricci curvature develops a Type-I singularity and shrinks to a round point. The same is true for Ricci flow on a compact four-manifold with positive curvature operator \cite{H86}. By the works of Hamilton \cite{H88} and Chow \cite{BCh91}, Ricci flow on $S^2$ with an arbitrary initial metric always develops a Type-I singularity and shrinks to a round point. On a compact $N$-dimensional manifold with $N\geq 3$, B\"{o}hm and Wilking \cite{BW08} proved that Ricci flow starting at a metric with $2$-positive curvature operator (i.e., the sum of its two smallest eigenvalues is positive) develops a Type-I singularity and shrinks to a round point. We note that this result in dimension $N\leq 4$ was known earlier \cites{H82, H86, Ch91}. Brendle \cite{B08} generalized the result of \cite{BW08} under a much weaker assumption on the curvature operator. All these Type-I singularities are \emph{global} in the sense that the volume of a manifold shrinks to zero at time $T$.

In \cite{H95}, Hamilton sketched intuitively the formation of \emph{local} singularities under Ricci flow. By local we mean that a singularity forms on a compact subset of a manifold while the volume of the manifold remains positive at time $T$. Rigorous results on finite time local singularities in Ricci flow were obtained later. On a noncompact warped product $\RR\times_f S^n$ ($n\geq 2$), Simon \cite{S00} showed that there are Ricci flow solutions that encounter finite time local singularities. For local singularities in the K\"{a}hler-Ricci flow, the first examples were constructed on holomorphic line bundles over $\mathbb{CP}^{n-1}$ ($n\geq 2$) using $U(n)$-invariant gradient shrinking K\"{a}hler-Ricci solitons \cite{FIK03}.

Hamilton's examples of local singularities are the so-called \emph{neckpinch} singularities. To describe them precisely, we recall the blow-up technique in singularity analysis. We say that a sequence $\{(x_i, t_i) \}_{i=0}^\infty$ of points and times in a Ricci flow is a \emph{blow-up sequence} at time $T$ if $t_i\nearrow T$ and $\left\vert\Rm(x_i,t_i)\right\vert\nearrow\infty$ as $i\nearrow\infty$. A blow-up sequence has a \emph{pointed singularity model} if the sequence of parabolically dilated metrics
\begin{align*}
 g_i(x, t) : = \left\vert\Rm(x_i,t_i)\right\vert g \left(x, t_i + \left\vert\Rm(x_i,t_i)\right\vert^{-1}t\right)
\end{align*}
has a complete smooth limiting metric. A Ricci flow solution is said to develop a neckpinch singularity at time $T<\infty$ if there is some blow-up sequence at $T$ whose pointed singularity model exists and is given by the self-similarly shrinking Ricci soliton on the cylinder $\RR\times S^n$ ($n\geq 2$).

A neckpinch singularity is called \emph{nondegenerate} if every pointed singularity model of any blow-up sequence at $T$ is the shrinking cylinder soliton. The first rigorous examples of finite time neckpinch singularities in Ricci flow on a compact manifold were produced by Angenent and Knopf \cite{AK04}. They exhibited a class of rotationally symmetric metrics on $S^{N}$ ($N\geq 3$) which develop Type-I neckpinch singularities under Ricci flow. In a subsequent paper \cite{AK07}, the same authors obtained the precise asymptotics for such neckpinch singularities.

A neckpinch singularity is said to be \emph{degenerate} if there is at least one blow-up sequence at $T$ with a pointed singularity model that is not the shrinking cylinder soliton. A degenerate neckpinch is expected to be a Type-$\II$ singularity. In this paper, we construct rotationally symmetric Ricci flow solutions on $\RR^N$ ($N \geq 3$) that encounter a finite time Type-$\II$ singularity. These are the first examples of degenerate neckpinches on a complete noncompact manifold in dimension three or higher. Moreover, we describe the asymptotic behavior of the geometry of such a solution near the first singular time $T$. Before stating our main theorem, we first recount the existing results concerning Type-$\II$ singularities in Ricci flow.

Daskalopoulos and Hamilton \cite{DH04} showed that on $\RR^2$ there exist complete Ricci flow solutions that form Type-$\II$ singularities at the rate $(T-t)^{-2}$. Their proof is particular to dimension two, in which case the metric is conformal to the Euclidean metric by $g = u (dx^2+dy^2)$ and the conformal factor $u$ evolves by the logarithmic fast diffusion equation $\p_t u = \Delta\log u$. Assuming rotational symmetry and additional constraints, Daskalopoulos and del Pino \cite{DdP07} gave a precise description of the extinction profile of such a maximal solution on $\RR^2$: up to proper scaling, it must be a cigar soliton in an inner region, and a logarithmic cusp in an outer region. Daskalopoulos and \v{S}e\v{s}um \cite{DS10} proved the same result without assuming rotational symmetry. An extension of the results of \cites{DdP07, DS10} was obtained by Hui \cite{Hui12}. The formal asymptotics of the extinction profile were derived by King \cite{K93}.

In dimension three or higher, if one is willing to assume rotational symmetry of the metrics, then the Ricci flow system of equations is reduced to a parabolic PDE for a scalar function. Gu and Zhu \cite{GZh08} proved the existence of Type-$\II$ singularities in Ricci flow on $S^{N}$ ($N \geq 3$), although their work shed little light on the geometric details of such solutions. Garfinkle and Isenberg \cites{GI05, GI08} have conducted numerical investigations on the formation of Type-$\II$ singularities on $S^3$ modeled by degenerate neckpinches.

In their recent works \cites{AIK11, AIK12}, Angenent, Isenberg, and Knopf demonstrated the existence of rotationally symmetric Ricci flow solutions on $S^{N}$ ($N\geq 3$) that develop finite time Type-$\II$ degenerate neckpinches. Their solutions become singular at the rate $(T-t)^{-2+2/k}$ for each integer $k\geq 3$. Moreover, they were able to describe the asymptotic profiles of these solutions. The techniques in \cites{AIK11, AIK12} have been applied to the asymptotic analysis of the formation of singularities in other geometric flows. For example, Angenent and Vel\'{a}zquez \cite{AV95} studied the asymptotic shape of cusp singularities in the curve shortening flow. The same authors \cite{AV97} constructed solutions with degenerate neckpinches to the mean curvature flow.

In this paper, we consider rotationally symmetric Riemannian metrics on $\RR^N$ with $N\geq 3$. We first note that Ricci flow on $\RR^N$ can encounter a finite time singularity under suitable initial conditions. For example, take a metric on $S^{N}$ as constructed in \cite{AK04} and conformally open up the north pole of the sphere. This produces an initial geometry on $\RR^N$, which one expects to develop a finite time Type-I neckpinch singularity under Ricci flow. Similarly, one expects that there are Ricci flow solutions that form a finite time Type-$\II$ singularity on $\RR^N$. Indeed, this happens on $\RR^2$ \cite{DH04}.

We now state the main result of this paper.
\begin{thm}\label{thm1}
In each dimension $N\geq 3$, for each real number $\lambda\geq 1$ and any pair of positive constants $A_1^-, A_1^+$ \footnote{The constants $A_1^-, A_1^+$ are defined in Lemma \ref{int_bar}.} with $A_1^->A_1^+$, there exists an open (in $C^2$ topology) set $\CMcal{G}_N$ of complete rotationally symmetric metrics on $\RR^N$ such that the Ricci flow starting at each $g_0\in\CMcal{G}_N$ has a unique solution $g(t)$ for $t\in[0,T)$, $T<\infty$. The solution $g(t)$ develops a finite time global singularity at time $T$ with the following properties.
\begin{enumerate}
\item The singularity is Type-$\II$ with
\begin{align*}
 \frac{C^+}{(T-t)^{\lambda+1}} \leq \sup\limits_{x\in\RR^N} \left\vert \Rm(x,t)\right\vert \leq \frac{C^-}{(T-t)^{\lambda+1}}
\end{align*}
attained at the origin for some constants $C^{\pm} = C(N, A_1^{\pm})$.
\item If one rescales the solution so that the distance from the origin dilates at the rate $(T-t)^{-(\lambda+1)/2}$, then the metric converges uniformly on intervals of order $(T-t)^{(\lambda+1)/2}$ to the Bryant soliton.
\item If one rescales the solution at the parabolic rate $(T-t)^{-1/2}$, then the metric converges uniformly to the shrinking cylinder soliton near spatial infinity.
\item The solution exhibits the asymptotic behavior of the formal solution described in Section \ref{formal}.
\end{enumerate}
\end{thm}

\begin{remark}
The singular time $T$ is determined only by the initial radius of the asymptotic cylinder at spatial infinity. In terms of the rescaled time $\tau_0$ (cf. Proposition \ref{barriers}), $T=e^{-\tau_0}$.
\end{remark}

Our result is inspired by the works \cites{AIK11, AIK12}. To prove this theorem, we begin by constructing, for each choice of $N\geq 3$ and $\lambda\geq 1$, a formal solution to Ricci flow on $\RR^N$ with the curvature blow-up rate of $(T-t)^{-(\lambda+1)}$ near the origin and that of $(T-t)^{-1}$ near spatial infinity. Using this formal solution, we then construct upper and lower barriers to the Ricci flow PDE, for which a comparison principle is also proved. Before the first singular time $T$, the curvatures are bounded and so the Ricci flow solution exists and is unique \cites{Shi89-1, ChZh06}. Thus, for any initial data between the barriers, we obtain unique complete solutions to Ricci flow whose asymptotic properties, by the comparison principle, are the same as those of the formal solution. 

Theorem \ref{thm1} is interesting in several aspects. First of all, this shows that Type-$\II$ singularities in Ricci flow on $\RR^N$ can occur arbitrarily slowly with curvatures blowing up at arbitrarily fast rate, complementing the existing works \cites{AIK11, AIK12, DH04}. The $\lambda=1$ case in our theorem can be viewed as a higher dimensional analogue of the result (restricted to the rotationally symmetric solutions) of \cite{DH04}. The asymptotics in Theorem \ref{thm1} can be compared to those in \cites{DdP07, DS10}. Secondly, solutions in \cite{AIK12} become singular at the set of discrete rates $(T-t)^{-2+2/k}$, where $k\in\NN$ and $k\geq 3$. In contrast, the curvature blow-up rates of the solutions in Theorem \ref{thm1} form a continuum since $\lambda\in[1,\infty)$. In particular, the $\lambda=1$ case can be thought of as the limiting case of the Main Theorem in \cite{AIK12} as $k\nearrow\infty$. Thirdly, the analysis in \cites{GI05, GI08, GZh08} suggests that the formation of Type-$\II$ singularities on a compact manifold is an unstable property. This is reflected in that the proof in \cite{AIK12} uses the somewhat indirect Wa\.{z}ewski retraction method. In contrast, we use a comparison principle to give a direct proof of Theorem \ref{thm1}. So one may regard the formation of Type-$\II$ singularities on a noncompact manifold to be a stable property.

Our solutions near the origin are modeled by the Bryant soliton. This is reasonable because blow-ups of Ricci flow singularities are expected, and in many cases proved, to be Ricci solitons. On $\RR^N$ ($N\geq 3$), the Bryant soliton is a rotationally symmetric gradient steady Ricci soliton with positive curvature operator \cites{Bryant, ChLN06}. The uniqueness (up to homothetic scaling) of the Bryant soliton is an interesting question. In dimension three, Bryant \cite{Bryant} showed that there are no other rotationally symmetric steady Ricci solitons; other non-rotationally symmetric solitons exist in higher dimensions \cite{I94}. Perelman \cite{P02} asked if any three-dimensional steady Ricci soliton is necessarily rotationally symmetric. Brendle \cite{B13} recently answered this question affirmatively, thereby establishing the uniqueness of the Bryant soliton in dimension three, under a non-collapsing assumption. Brendle \cite{B12-2} also proved a higher dimensional version of his theorem. Other uniqueness result for the Bryant soliton holds under an additional assumption such as local conformal flatness \cite{CCh12}, or suitable asymptotics near spatial infinity \cite{B11}, or half conformal flatness \cite{ChW11}.

In \cite{P03-1}, Perelman described a special family of Ricci flow solutions on $\RR^3$, the so-called \emph{standard solutions}, which are complete rotationally symmetric metrics asymptotic to a round cylinder at spatial infinity. The standard solutions are used to construct long-time solutions and to study Ricci flow with surgery \cites{P03-1, KL08}. Our next result shows that a subset of the Ricci flow solutions in Theorem \ref{thm1} are in fact standard solutions in the sense defined in \cite{LT}. We refer the reader to Section \ref{stansolsect} for the precise definition of a standard solution.

\begin{thm}\label{thm2}
Let $\CMcal{G}_N$ be given as in Theorem \ref{thm1}. There exists an open (in $C^6$ topology) set $\CMcal{G}_N^\ast \subset \CMcal{G}_N$ such that the Ricci flow starting at each $g_0\in\CMcal{G}_N^\ast$ has a unique standard solution $g(t)$ on $\RR^N$ for $t\in[0,T)$, $T<\infty$. Moreover, the solution $g(t)$ satisfies all the properties described in Theorem \ref{thm1}.
\end{thm}

Bennett Chow and Gang Tian have conjectured \cite{Lu} the following.
\begin{conj}[Chow-Tian]
Sequences of appropriately scaled standard solutions (as defined in \cite{LT}) with marked origins converge to the Bryant soliton in a suitable sense.
\end{conj}

A corollary of Theorem \ref{thm2} gives evidence in favor of the Chow-Tian conjecture.
\begin{cor}\label{cor1}
In dimension $N\geq 3$, there exist standard solutions to Ricci flow whose blow-ups near the origin converge uniformly to the Bryant soliton.
\end{cor}

This paper is organized as follows. In Section 2, we describe the basic set-up and write down the parabolic PDE for the Ricci flow assuming rotational symmetry. In Section 3, we construct the formal solution with matched asymptotics. The formal solution is then used to construct the subsolution and supersolution to the PDE in Section 4. In Section 5, we build lower and upper barriers to the PDE from the subsolution and supersolution, respectively. In Section 6, we prove a comparison principle for the PDE and give the proof of Theorem \ref{thm1}. In Section 7, we relate the solutions we construct to the standard solutions and prove Theorem \ref{thm2}.


\subsection*{Acknowledgments} This work is part of my Ph.D. thesis at the University of Texas at
Austin. I sincerely thank my adviser Prof. Dan Knopf for his mentorship and support, without which this work would not have been possible. As usual, I appreciate Dan's good humor. I also thank the referee for the meticulous review of and many constructive comments on the manuscript.


\section{Preliminaries}
From now on, we set $n:=N-1\geq 2$. We puncture $\RR^N$ at the origin and identify the remaining manifold with $(0,\infty)\times S^{n}$. For $x\in(0,\infty)$, we define a warped product metric
\begin{align*}
 g := g(t,x) = \phv^2(t,x) dx^2 + \psi^2(t,x) g_{\text{sph}},
\end{align*}
where $g_{\text{sph}}$ is the metric of constant sectional curvature one on $S^n$.

The distance $s$ to the origin is
\begin{align*}
 s(t,x):=\int_0^x \phv(t,y)dy.
\end{align*}
In the $s$-coordinate, the metric becomes
\begin{align}\label{eq:g(s)}
 g = ds^2 + \psi^2\left(s, t\right)g_{\text{sph}}.
\end{align}
Extending the metric $g$ to a smooth complete rotationally symmetric metric, which we still denote by $g$, on $\RR^N$, then $\psi$ necessarily satisfy the boundary conditions
\begin{align*}
 \lim\limits_{x\searrow 0} \psi = 0, \quad \quad \lim\limits_{x\searrow 0} \psi_s = 1.
\end{align*}

We use the notation $\left.\p_t\right\vert_{\cdot}$ for taking the time derivative while keeping the quantity ``$\cdot$'' fixed. Then (cf. \cite{AIK12}*{Section 2.1})
\begin{align*}
 \left[\left.\p_t\right\vert_x, \p_s \right] & = -n\frac{\psi_{ss}}{\psi}\p_s.
\end{align*}
In the $s$-coordinate, the Ricci flow system is reduced to the parabolic equation for $\psi$:
\begin{align}\label{eq:psi}
 \left.\p_t\right\vert_x \psi & = \psi_{ss}-(n-1)\frac{1-\psi_s^2}{\psi}.
\end{align}
The function $\phv$, which is suppressed in the $s$-coordinate, evolves under Ricci flow by
\begin{align*}
 \left.\p_t\right\vert_x \log\phv & = n\frac{\psi_{ss}}{\psi}.
\end{align*}

Let $K$ be the sectional curvature of a two-plane orthogonal to the sphere $\{x\}\times S^n$ and $L$ be the sectional curvature of a tangential two-plane. Then
\begin{align*}
 K &= -\frac{\psi_{ss}}{\psi}, \quad L = \frac{1-\psi_s^2}{\psi^2}.
\end{align*}
In particular, $\left\vert \Rm\right\vert^2 = 2n K^2 + n(n-1) L^2$.

Since the metric is smooth and $\lim\limits_{x\searrow 0} \psi_s = 1$, we must have $\psi_s>0$ in a neighborhood of the origin. So we can use $\psi$ as a new coordinate near the origin, writing
\begin{align}\label{eq:g(z)}
 g = z(\psi, t)^{-1} d\psi^2 + \psi^2 g_{\text{sph}},
\end{align}
where $z\left(\psi,t \right) := \psi_s^2$. Then the sectional curvatures are given by
\begin{align*}
 K &= -\frac{z_\psi}{2\psi}, \quad L = \frac{1-z}{\psi^2}.
\end{align*}

Under Ricci flow, the metric \eqref{eq:g(z)} evolves by (see \cite{AIK12}*{Section 2.2})
\begin{align}\label{eq:RFpsi}
\left.\p_t\right\vert_\psi z = \mathcal{E}_{\psi}[z],
\end{align}
where $\mathcal{E}_{\psi}$ is the purely local quasilinear operator
\begin{align*}
 \mathcal{E}_{\psi}[z] & := zz_{\psi\psi}-\frac{1}{2}z_{\psi}^2+(n-1-z)\frac{z_\psi}{\psi}+2(n-1)\frac{(1-z)z}{\psi^2}.
\end{align*}
We can split $\mathcal{E}_{\psi}$ into a linear and a quadratic term:
\begin{align*}
 \mathcal{E}_\psi[z] = \mathcal{L}_\psi[z] + \mathcal{Q}_\psi[z],
\end{align*}
where
\begin{align}
 \mathcal{L}_\psi[z] &:= (n-1)\left(\frac{z_\psi}{\psi} + 2\frac{z}{\psi^2}  \right), \label{eq:L}\\ 
 \mathcal{Q}_\psi[z] &:= z z_{\psi\psi} - \frac{1}{2} z_\psi^2 - \frac{z z_\psi}{\psi} - 2(n-1)\frac{z^2}{\psi^2}. \label{eq:Q}
\end{align}
The quadratic part defines a symmetric bilinear operator
\begin{align}\label{eq:hatQ}
\hat{\mathcal{Q}}_{\psi}[z_1,z_2] & :=  \frac{1}{2}\left[ z_1 (z_2)_{\psi\psi} + z_2 (z_1)_{\psi\psi} - (z_1)_\psi(z_2)_\psi\right] - \frac{z_1 (z_2)_\psi + z_2 (z_1)_\psi}{2\psi} - 2(n-1)\frac{z_1 z_2}{\psi^2}. \notag
\end{align}
In particular, $\mathcal{Q}_\psi[z]=\hat{\mathcal{Q}}_\psi[z,z]$.

Throughout this paper, we use $C_k$ ($k\in\NN$) to denote a positive constant that may change from line to line. The expression ``$f\lesssim g$'' means $f\leq C_k g$ for some constant $C_k$.


\section{The formal solution}\label{formal}
We first briefly review the formal solution (i.e., approximate solution) in \cites{AIK11, AIK12}. Introducing the coordinates consistent with a parabolic cylindrical blow-up:
\begin{align}
 u:=\frac{\psi}{\sqrt{2(n-1)(T-t)}},\quad \sigma:=\frac{s}{\sqrt{T-t}},\quad \tau:=-\log(T-t),
\end{align}
then in these coordinates, equation \eqref{eq:psi} becomes
\begin{align}\label{eq:RF_sigma}
 \left.\p_\tau\right\vert_\sigma u = u_{\sigma\sigma} - \left(\frac{\sigma}{2} + nI[u]\right)u_{\sigma}+\frac{u-u^{-1}}{2}+(n-1)\frac{u^2_\sigma}{u},
\end{align}
where
\begin{align*}
 I[u](\sigma, \tau) := \int_0^\sigma \frac{u_{\hat\sigma\hat\sigma}(\hat\sigma,\tau)}{u(\hat\sigma,\tau)} d\hat\sigma.
\end{align*}
For bounded $\sigma$, the solution to equation \eqref{eq:RF_sigma} is approximated by
\begin{align*}
 u \approx 1 + \sum\limits_{m=0}^\infty a_m e^{(1-m/2)\tau} h_m(\sigma),
\end{align*}
where $h_m$ is the $m$-th Hermite polynomial. In \cites{AIK11, AIK12}, the authors assume a nondegenerate neckpinch occurs at the equator of $S^{N}$ in such a way that the Ricci flow solution \emph{does not} approach a cylinder too quickly. So the term with $m=k$ is dominant for some specified $k\geq 3$. They then construct a formal solution using matched asymptotics in four connected regions: the outer, parabolic, intermediate, and tip regions. Their construction starts in the parabolic region which models a nondegenerate neckpinch near the equator of $S^{N}$, and ends in the tip region which models a degenerate neckpinch at one of the poles of $S^{N}$.

In this paper, we are interested in solutions that approach a cylinder ``quickly''. This leads us to the following construction. We first build a model for a degenerate neckpinch near the origin of $\RR^N$, we then work our way out to the rest of the manifold. We will see that our formal solution is defined in two connected regions: the \emph{interior} and \emph{exterior} regions. It turns out that, cf. the proof of Lemma \ref{complete}, these two regions are enough to define a complete metric on $\RR^N$. One may compare this to the compact case and conjecture that in the noncompact case the parabolic and outer regions are pushed to spatial infinity.


\subsection{Approximate solution in the interior region}
In the interior region, which will be defined in \eqref{eq:def_int}, we expect $u$ to be small and introduce the variable
\begin{align*}
 r := e^{\gamma\tau} u ,
\end{align*}
where $\gamma>0$ is a constant to be specified.

In the $u$-coordinate, by the change of variable formulae
\begin{align*}
 \left.\p_t\right\vert_\psi z  & = \left\{\left. \p_{\tau}\right\vert_u z + z_{u}\left(\left.\p_\tau\right\vert_{\psi}u\right) \right\}\frac{d\tau}{d t}=\left(\left.\p_{\tau}\right\vert_u z+\frac{1}{2}\psi z_{\psi} \right)e^{\tau},\\
 \mathcal{E}_{\psi}[z] & = \frac{1}{2(n-1)}e^{\tau} \mathcal{E}_{u}[z],
\end{align*}
equation \eqref{eq:RFpsi} becomes
\begin{align}\label{eq:RFzu}
 \left.\p_\tau\right\vert_u z &=\frac{1}{2(n-1)}\mathcal{E}_u[z]-\frac{1}{2}u z_u,
\end{align}
where we have used $\psi z_{\psi} = u z_u$.

In the $r$-coordinate, since
\begin{align*}
 \left.\p_\tau\right\vert_r z  & = \left.\p_\tau\right\vert_u z  + z_u\left(\left.\p_{\tau}\right\vert_{r} u \right) = \left.\p_\tau\right\vert_u z - \gamma u z_u ,\\ 
 \mathcal{E}_{u}[z] & = e^{2\gamma\tau} \mathcal{E}_{r}[z],\\
\end{align*}
equation \eqref{eq:RFzu} becomes, after using $u z_u = r z_r$,
\begin{align*}
 \mathcal{T}_r[z]=0,
\end{align*}
where
\begin{align}\label{eq:T_r}
 \mathcal{T}_r[z] := e^{-2\gamma \tau}\left\{\left.\p_\tau\right\vert_r z + \left(\frac{1}{2}+\gamma\right)r z_r\right\} - \frac{1}{2(n-1)}\mathcal{E}_r[z].
\end{align}
For sufficiently large $\tau$, the term involving $e^{-2\gamma\tau}$ is negligible and the equation $\mathcal{T}_r[z]=0$ is approximated by the equation
\begin{align*}
 \mathcal{E}_r[\tilde z]=0,
\end{align*}
whose solution, subject to the boundary conditions $\tilde z(0)=1$ and $\tilde z(\infty)=0$, is a Bryant soliton profile function
\begin{align*}
 \tilde z(r) = \mathfrak{B}(ar),
\end{align*}
where $a>0$ is an arbitrary constant. Each member of the one-parameter family of complete smooth metrics given by
\begin{align*}
 g = \mathfrak{B}^{-1}(ar) dr^2 + r^2 g_{\text{sph}}
\end{align*}
is a scaled version of the Bryant soliton.

The function $ \mathfrak{B}(r)$ is smooth and strictly monotone decreasing for all $r>0$. Near $r=0$, $ \mathfrak{B}(r)$ has the asymptotic expansion
\begin{align}\label{eq:Br0}
 \mathfrak{B}(r) & = 1 - b_2 r^2 + b_3 r^4 +b_3 r^6 + \cdots\quad\text{as } r\searrow 0,
\end{align}
where $b_k$'s are constants determined by $b_2$. In particular, $b_2>0$. Near $r=\infty$, $\mathfrak{B}(r)$ has the asymptotic expansion
\begin{align}\label{eq:Bri}
\mathfrak{B}(r) & = c_1r^{-2} + c_2 r^{-4} + c_3 r^{-6} + \cdots\quad\text{as } r\nearrow \infty,
\end{align}
where $c_k$'s are constants determined by $c_1$. In this paper, we normalize $\mathfrak{B}(r)$ by setting $c_1=1$ (so $b_2$ is also fixed). For more information on $\mathfrak{B}(r)$, we refer the reader to \cite{AIK11}*{Appendix B}.

We refine the approximate solution by considering an expansion of the form
\begin{align}
 z & \approx \mathfrak{B}(ar) + e^{-\tilde\gamma\tau}\beta_1(r)+e^{-\tilde\gamma\tau}\beta_2(r)+\cdots
\end{align}
for some $\tilde\gamma>0$, then
\begin{align*}
 z \approx a^{-2}r^{-2}\quad\text{as } r\nearrow\infty \text{ for } \tau \text{ large},
\end{align*}
which, in terms of the $u$-coordinate, is
\begin{align*}
 z \approx a^{-2}e^{-2\gamma\tau}u^{-2}\quad \text{as }\tau\nearrow\infty \text{ for } u \text{ small}.
\end{align*}
In Lemma \ref{z_int}, we will use $\mathfrak{B}(ar)\pm e^{-\tilde\gamma\tau}\beta_1(r)$ with $\tilde\gamma=\lambda$ to construct subsolution and supersolution in the interior region.


\subsection{Approximate solution in the exterior region}
We expect the exterior region, which is to be specified in \eqref{eq:def_ext}, to be a time-dependent subset of the neighborhood of the origin where $1>z>0$ and $0<u<1$. In this region, $z$ evolves by equation \eqref{eq:RFzu}, i.e.,
\begin{align*}
 \left.\p_\tau\right\vert_u z &=\frac{1}{2(n-1)}\mathcal{E}_u[z]-\frac{1}{2}u z_u.
\end{align*}
To construct an approximate solution to this equation, we use the series ansatz
\begin{align}\label{eq:series}
 z & =  \sum\limits_{m = 1}^{\infty} e^{-m\lambda\tau} Z_m(u),
\end{align}
where $\lambda>0$ is a constant to be chosen. We substitute \eqref{eq:series} into the PDE above and split $\mathcal{E}_u[z]$ into the linear and quadratic parts as given in \eqref{eq:L} and \eqref{eq:Q} respectively. By comparing the coefficients of $e^{-m\lambda\tau}$ in the resulting equation, we find $Z_m$ must satisfy the first order ODE
\begin{align}\label{eq:ODE_Z}
 \frac{1}{2}\left(u^{-1}-u\right)\frac{d Z_m}{du} + \left(u^{-2}+m\lambda\right)Z_m & = -\frac{1}{2(n-1)}\sum\limits_{i=1}^{m-1}\hat{\mathcal{Q}}_u[Z_i,Z_{m-i}].
\end{align}

When $m=1$, equation \eqref{eq:ODE_Z} is a linear homogeneous equation
\begin{align}\label{eq:ODE_Z1}
\frac{1}{2} \left(u^{-1}-u \right)\frac{d Z_1}{du} + \left(u^{-2}+\lambda\right)Z_1 & = 0,
\end{align}
whose solutions are
\begin{align}
 Z_1(u) = b u^{-2}\left(1-u^2\right)^{1+\lambda},
\end{align}
where $b$ is an arbitrary constant that should satisfy a matching condition, cf. Subsection \ref{matching} and Lemma \ref{patch}.

When $m=2$, equation \eqref{eq:ODE_Z} becomes
\begin{align}\label{eq:ODE_Z2}
 \frac{1}{2}\left(u^{-1}-u\right)\frac{d Z_2}{du} + \left(u^{-2}+2\lambda\right)Z_2 & = \mathcal{Q}_u[Z_1],
\end{align}
where
\begin{align}\label{eq:QuZ1}
\mathcal{Q}_u[Z_1]  = 2b^2u^{-6}\left(1-u^2\right)^{2\lambda}\left\{ 4 - n\left(1-u^2\right)^2 + 2u^2\left(\lambda-3\right)+ u^4\left(\lambda-1\right)^2 \right\}.
\end{align}
The solutions of equation \eqref{eq:ODE_Z2} are 
\begin{align*}
 Z_2(u) &= u^{-2}\left(1-u^2\right)^{1+2\lambda} f(u),
\end{align*}
where
\begin{align*}
f(u) & = C_1 -2b^2\left(\frac{4-n}{u^2}  - \frac{\lambda^2 - 1}{1-u^2}\right) - 4(1+\lambda)b^2 \left\{\log\left(1-u^2\right) - 2\log u \right\}
\end{align*}
for an arbitrary constant $C_1$.

By direction computation, we have the following.
\begin{lemma}
If $\lambda\geq 1$, then
\begin{align*}
 \lim\limits_{u\nearrow 1}\left\vert \frac{Z_2(u)}{Z_1(u)}\right\vert = 0.
\end{align*}
\end{lemma}
So for any $\lambda\geq 1$,
\begin{align*}
 z(u,\tau) &\approx e^{-\lambda\tau} b  u^{-2}\left(1-u^2\right)^{1+\lambda} + O\left(e^{-2\lambda\tau}Z_2(u)\right)
\end{align*}
is a valid approximation as $u\nearrow 1$ when $\tau$ is sufficiently large. Going in the other direction, as $u\searrow 0$,
\begin{align*}
 Z_1(u) \approx b u^{-2},\quad  Z_2(u) = O\left(u^{-4}\right),
\end{align*}
so the approximation
\begin{align*}
 z &\approx e^{-\lambda\tau}b u^{-2} + O\left(e^{-2\lambda\tau}u^{-4}\right).
\end{align*}
is valid as long as
\begin{align*}
 \left\vert e^{-\lambda\tau}b u^{-2} \right\vert \gg \left\vert e^{-2\lambda\tau}u^{-4} \right\vert,
\end{align*}
which is when
\begin{align*}
 u \gg e^{-\lambda\tau/2},
\end{align*}
or equivalently, in the $r$-coordinate,
\begin{align*}
 r \gg e^{(\gamma-\lambda/2)\tau}.
\end{align*}

From now on, given $\lambda\geq 1$, we choose $\gamma=\lambda/2$.


\subsection{Matching condition}\label{matching}
We now match the approximate solutions in the interior and exterior regions assuming $\tau$ is sufficiently large. At $r=A \gg 1$, the approximate solution in the interior region is
\begin{align*}
 z(A) & \approx \mathfrak{B}(aA)\\
      & \approx (aA) ^{-2}.
\end{align*}
At $u= e^{-\lambda\tau/2} A$, the approximate solution in the exterior region is
\begin{align*}
 z\left( e^{-\lambda\tau/2} A \right) & \approx e^{-\lambda\tau} Z_1\left(e^{-\lambda\tau/2} A \right)\\
& \approx b A^{-2}\left(1-e^{-\lambda\tau}A^2\right)^2\\
&\approx bA^{-2}.
\end{align*}
Thus, matching these two approximations implies that for a given constant $a>0$, we ought to have
\begin{align*}
 b \approx a^{-2}.
\end{align*}
This relation is made precise in Lemma \ref{patch}.


\subsection{Features of the formal solution} For each $N\geq 3$ and $\lambda\geq 1$, we match the approximate solutions in the interior and the exterior regions to obtain the formal solution, denoted by $z_{\text{form}}$. We will, cf. Lemma \ref{complete}, use the formal solution to define metrics that are complete on $\RR^N$, i.e., one approaches spatial infinity as $u\nearrow 1$. As $u\nearrow 1$, $z_{\text{form}}(u)\searrow 0$, i.e., $\psi_s\searrow 0$, so the metric \eqref{eq:g(s)} is approaching that of a round cylinder near spatial infinity. As $u\searrow 0$ and $\tau\nearrow\infty$, $z_{\text{form}}(u)\nearrow 1$ and the formal solution $z_{\text{form}}$ is asymptotic to a Bryant soliton profile function near the origin.

The norm of the curvature tensor for the formal solution achieves its maximum value at the origin $O$ \cite{AIK12}, where we have
\begin{align*}
 \left\vert \Rm(O,t) \right\vert & = \frac{C}{(T-t)^{\lambda+1}}
\end{align*}
for some constant $C$ depending on the dimension $N$ and the scaling parameter $a$ in $\mathfrak{B}(ar)$. Thus, the curvature of a Ricci flow solution that asymptotically approaches the formal solution necessarily blows up at the same rate.


\section{Subsolution and supersolution}\label{subsuper}
Let $g$ be a metric of the form \eqref{eq:g(z)} that is evolving under Ricci flow. Then $g$ is determined by its profile function $z$ which, in the $u$-coordinate, satisfies the quasilinear parabolic PDE \eqref{eq:RFzu}. In this section, we construct subsolution and supersolution to this equation in the interior and exterior regions, respectively.


\subsection{In the interior region}
Recall that in the $r$-coordinate, $z$ satisfies the equation $\mathcal{T}_r[z]=0$, where the operator $\mathcal{T}_r$ is defined in \eqref{eq:T_r}. We call $z$ a subsolution (supersolution) of $\mathcal{T}_r[z]=0$ if $\mathcal{T}_r[z]\leq0$ ($\geq 0$).

\begin{lemma}\label{z_int}
Let $\lambda\geq 1$. For any $A_1>0$, there exist a bounded function $\beta:(0,\infty)\to\RR$, a sufficiently small $B_1>0$, and a sufficiently large $\tau_1<\infty$, all depending only on $A_1$ such that the functions
\begin{align}
 z_{\text{int}}^{\pm} & := \mathfrak{B}(A_1 r) \pm e^{-\lambda\tau}\beta(r)
\end{align}
are sub- ($z_{\text{int}}^{-}$) and super- ($z_{\text{int}}^{+}$) solutions of  $\mathcal{T}_r[z]=0$ in the interior region
\begin{align}\label{eq:def_int}
 \Omega_{\text{int}} := \left \{ 0 \leq r \leq B_1 e^{\lambda\tau/2} \right \}
\end{align}
for all $\tau \geq \tau_1$.
\end{lemma}
\begin{proof}
Let $\mathbf{B}(r):=\mathfrak{B}(A_1 r)$. For $z(r) = \mathbf{B}(r) + e^{-\lambda\tau}\beta(r)$ to be a supersolution, it suffices to show $\mathcal{T}_r [z]\geq 0$.

Since $\mathbf{B}(r)$ solves $\mathcal{E}_r[z]=0$,
\begin{align*}
\mathcal{T}_r \left[z^{+}_{\text{int}}\right]   = & e^{-\lambda\tau}\left\{ - \frac{\mathcal{L}_r[\beta] + 2\hat{\mathcal{Q}}_r[\mathbf{B}, \beta]}{2(n-1)} + \frac{\lambda+1}{2}r\mathbf{B}'\right\} + e^{-2\lambda\tau}\left\{-\lambda\beta + \frac{\lambda+1}{2}r\beta' - \frac{\mathcal{Q}_r[\beta]}{2(n-1)} \right\}.
\end{align*}
Set $\hat A:=1+\frac{\lambda+1}{2}$, and let $\beta$ solve the second order linear inhomogeneous ODE
\begin{align}\label{eq:ODE1}
  \mathcal{L}_r[\beta] + 2\hat{\mathcal{Q}}_r[\mathbf{B}, \beta] = 2(n-1)\hat A r\mathbf{B}'.
\end{align}
Using the definitions of $\mathcal{L}_r$ and $\hat{\mathcal{Q}}_r$ in \eqref{eq:L} and \eqref{eq:Q} respectively, equation \eqref{eq:ODE1} becomes
\begin{align}\label{eq:ODE1-1}
 \mathbf{B} \beta''  + & \left\{\frac{n-1}{r} - \mathbf{B}' - \frac{\mathbf{B}}{r} \right\}\beta' + \left\{\mathbf{B}'' - \frac{\mathbf{B}'}{r} + 2(n-1)\frac{1-2\mathbf{B}}{r^2} \right\}\beta = 2(n-1)\hat{A}r\mathbf{B}'.
\end{align}
Recall the asymptotic expansions of $\mathbf{B}(r)$ near $r=0$ and $r=\infty$ given in \eqref{eq:Br0} and \eqref{eq:Bri}, respectively. Then near $r=0$, equation \eqref{eq:ODE1-1} is approximated by
\begin{align*}
 \beta'' + \frac{n-2}{r}\beta' - \frac{2(n-1)}{r^2}\beta = - C_1 r^2 \quad(C_1>0),
\end{align*}
whose general solution is
\begin{align*}
 \beta_0 = C_2 r^{1-n} + C_3 r^2 - C_4 r^4,
\end{align*}
where $C_2, C_3$ are arbitrary constants and $C_4$ is a constant depending on $C_1$. Discarding the unbounded solution and choosing $C_3=1$, then there exists a solution $\beta_p$ to equation \eqref{eq:ODE1} with
\begin{align*}
 \beta_p (r) = r^2 + o\left(r^2\right) \quad \text{as }r\searrow 0.
\end{align*}
Near $r=\infty$, the ODE \eqref{eq:ODE1-1} is a perturbation of the equation
\begin{align*}
 \frac{1}{(A_1 r)^2}\beta'' + \frac{n-1}{r} \beta' + \frac{2(n-1)}{r^2}\beta = - \frac{4(n-1)\hat A}{(A_1 r)^2},
\end{align*}
whose general solution is 
\begin{align*}
 \beta_\infty(r) & = C_5 r e^{-\alpha r^2} + C_6 r \int_{1}^{r}\rho^{-2}e^{-\alpha(r^2-\rho^2)}d\rho-\frac{2\hat A}{A_1^2}
\end{align*}
with $\alpha:=\frac{n-1}{2}A_1^2$ and arbitrary constants $C_5, C_6$. The second term in this expression is $O\left(r^{-2}\right)$. So every solution of equation \eqref{eq:ODE1}, in particular $\beta_p(r)$ given above, has the following asymptotic expansions:
\begin{align}\label{beta_asymp}
 \beta(r) = \left\{
\begin{array}{cc}
 r^2 + o\left(r^2\right) & \text{as }r\searrow 0,\\
 -2\hat A/A_1^2+o\left(1\right) & \text{as }r\nearrow \infty.
\end{array}
\right.
\end{align}
Also, the asymptotic expansions
\begin{align*}
 -r\mathbf{B}'(r) = \left\{
\begin{array}{cc}
C_7 r^2 + o\left(r^2\right) & \text{as } r\searrow 0, \\
C_8 r^{-2} + o\left(r^{-2}\right) & \text{as } r\nearrow \infty,
\end{array}\right.
\end{align*}
imply that
\begin{align*}
  -r\mathbf{B}'(r) \geq C_9 \min \left\{ r^2, r^{-2} \right\}.
\end{align*}
Then in view of \eqref{beta_asymp}, we have for $0 < r\leq 1$,
\begin{align*}
 \left\vert -\lambda \beta + \frac{\lambda+1}{2}r\beta' - \frac{\mathcal{Q}_r[\beta]}{2(n-1)} \right\vert\leq C_{10} r^2,
\end{align*}
and hence
\begin{align*}
 \mathcal{T}_r \left[z^{+}_{\text{int}}\right] & \geq -e^{-\lambda\tau}r\mathbf{B}'(r) - e^{-2\lambda\tau}C_{10} r^2\\
& \geq e^{-\lambda\tau}r^2 \left( C_9 -e^{-\lambda\tau} C_{10} \right)\\
& > 0
\end{align*}
for all $\tau\geq\tau_1$ with $\tau_1$ sufficiently large. And for $r\geq 1$, 
\begin{align*}
 \left\vert -\lambda \beta + \frac{\lambda+1}{2}r\beta' - \frac{\mathcal{Q}_r[\beta]}{2(n-1)} \right\vert\leq C_{11},
\end{align*}
so then
\begin{align*}
 \mathcal{T}_r \left[z^{+}_{\text{int}} \right] & \geq -e^{-\lambda\tau}r\mathbf{B}'(r) - e^{-2\lambda\tau}C\\
 & \geq e^{-\lambda\tau} \left( C_9 r^{-2} - e^{-\lambda\tau} C_{11} \right)\\
& > 0
\end{align*}
provided that $r<B_1 e^{\lambda\tau/2}$ with constant $B_1:=\sqrt{C_9/C_{11}}$.

Therefore, $z_{\text{int}}^+$ is indeed a supersolution. That $z_{\text{int}}^-$ is a subsolution is proved similarly.
\end{proof}


\subsection{In the exterior region}
Recall that in the $u$-coordinate, $z$ evolves by equation \eqref{eq:RFzu}, which we rewrite as $\mathcal{D}_u[z]=0$, where
\begin{align}\label{eq:D}
\mathcal{D}_u[z] & := \left.\p_\tau\right\vert_u z - \frac{1}{2}\left(u^{-1}-u\right)z_u-u^{-2}z-\frac{\mathcal{Q}_u[z]}{2(n-1)}.     
\end{align}
We call $z$ a subsolution (supersolution) of $\mathcal{D}_u[z]=0$ if $\mathcal{D}_r[z]\leq0$ ($\geq0$).
\begin{lemma}\label{z_ext}
Let $\lambda\geq 1$. Define $Z_1(u) := u^{-2}\left(1-u^2\right)^{1+\lambda}$. Given $A_2>0$, there exist a function $\zeta:(0,1)\to\RR$, a constant $B_2>0$, a sufficiently large $\tau_2<\infty$, and a constant $A_3^\ast<\infty$ depending only on $A_2$ such that for any $A_3\geq A_3^\ast$, the functions
\begin{align}
 z_{\text{ext}}^{\pm}(u,\tau):=e^{-\lambda\tau}A_2Z_1(u) \pm e^{-2\lambda\tau}A_3 \zeta(u)
\end{align}
are sub- ($z_{\text{ext}}^{-}$) and super- ($z_{\text{ext}}^{+}$) solutions of $\mathcal{D}_u[z]=0$ in the exterior region
\begin{align}\label{eq:def_ext}
 \Omega_{\text{ext}} := \left\{ B_2\sqrt{\frac{A_3}{A_2}}e^{-\lambda\tau/2} \leq u < 1 \right\},
\end{align}
for all $\tau\geq \tau_2$ where $\tau_2$ depends only on $A_2$ and $A_3$.
\end{lemma}
\begin{proof}
Since $A_2Z_1$ is a solution of the ODE \eqref{eq:ODE_Z1}, we have
\begin{align*}
 e^{2\lambda\tau}\mathcal{D}_u[z_{\text{ext}}^+] & =  A_3\left\{-\frac{1}{2}\left(u^{-1}-u\right)\zeta'-\left(u^{-2}+2\lambda\right)\zeta  \right\} - \frac{A_2^2}{2(n-1)}\mathcal{Q}_u[Z_1]\\
&\quad - \frac{A_2A_3}{n-1}e^{-\lambda\tau}\hat{\mathcal{Q}}_u[Z_1,\zeta] - \frac{A_3^2}{n-1}e^{-2\lambda\tau}\mathcal{Q}_u[\zeta].
\end{align*}
Since $0<u<1$, the definition of $Z_1$ implies that
\begin{align}\label{eq:Z1_deriv}
 \left\vert Z_1' \right\vert\leq \frac{C_1}{u\left(1-u^2\right)}Z_1,\quad \left\vert Z_1'' \right\vert \leq \frac{C_2}{u^2\left(1-u^2\right)^2}Z_1,
\end{align}
and from \eqref{eq:QuZ1},
\begin{align}\label{QZ1}
 \left\vert\mathcal{Q}_u\left[Z_1\right]\right\vert \leq C_3 u^{-6}\left(1-u^2\right)^{2\lambda}.
\end{align}

Let $\zeta : (0,1)\to\RR$ be a solution of the inhomogeneous ODE
\begin{align}\label{eq:ODE_zeta}
 -\frac{1}{2}\left(u^{-1}-u\right)\zeta'-(u^{-2}+2\lambda)\zeta = u^{-6}\left(1-u^2\right)^{2\lambda}.
\end{align}
Then we solve the ODE to obtain
\begin{align*}
\zeta(u) = u^{-4} \left(1-u^2\right)^{2\lambda}h(u),
\end{align*}
where
\begin{align*}
h(u) & = 1-2u^2+C_4 u^2\left(1-u^2\right) + 2u^2\left(1-u^2\right)\left[\log\left(1-u^2\right) - 2\log u\right]
\end{align*}
for an arbitrary constant $C_4$. This implies that $\zeta$ has the asymptotic behavior
\begin{align}\label{eq:zeta_asymp}
\zeta(u) = \left\{
\begin{array}{cc}
u^{-4} + O\left(u^{-2}\log u\right)& \text{as } u\searrow 0, \\
-\left(1-u^2\right)^{2\lambda} + O\left(\left(1-u^2\right)^{1+2\lambda}\log\left(1-u^2\right)\right)& \text{as } u\nearrow 1.
\end{array}\right. 
\end{align}
We then have the following estimates. For $0<u<1/2$, 
\begin{align}\label{eq:est1}
 \left\vert\hat{\mathcal{Q}}_u[Z_1,\zeta] \right\vert\leq C_5 u^{-8},\quad \left\vert\mathcal{Q}_u[\zeta] \right\vert\leq C_6 u^{-10}.
\end{align}
For $1/2\leq u<1$,
\begin{align}\label{eq:est2}
 \left\vert\hat{\mathcal{Q}}_u\left[Z_1,\zeta\right] \right\vert\leq C_7 \left(1-u^2\right)^{3\lambda-1},\quad \left\vert\mathcal{Q}_u\left[\zeta\right] \right\vert\leq C_8 \left(1-u^2\right)^{4\lambda-2}.
\end{align}
Using the definition of $\zeta$ and the estimate \eqref{QZ1}, we have
\begin{align*}
  e^{2\lambda\tau}\mathcal{D}_u[z_{\text{ext}}^+] & \geq \left(A_3 - C_3 A_2^2 \right)u^{-6}\left(1-u^2\right)^{2\lambda} - \frac{A_2A_3}{n-1}e^{-\lambda\tau}\left\vert \hat{\mathcal{Q}}_u\left[Z_1,\zeta\right] \right\vert - \frac{A_3^2}{n-1}e^{-2\lambda\tau} \left\vert \mathcal{Q}_u\left[\zeta \right]\right\vert.
\end{align*}
We choose $A_3^\ast = 2C_3A_2^2$. Then for $A_3\geq A_3^\ast$, we have the following. For $0<u\leq 1/2$, there exists a constant $B_2<\infty$ such that \eqref{eq:est1} implies
\begin{align*}
 e^{2\lambda\tau}\mathcal{D}_u\left[z_{\text{ext}}^+ \right] & \geq C_9 u^{-6}\left(A_2^2-C_5A_2A_3u^{-2}e^{-\lambda\tau}-C_6A_3^2u^{-4}e^{-2\lambda\tau} \right)\\
& \geq 0
\end{align*}
provided that $e^{\lambda\tau}u^2\geq B_2^2A_3/A_2$, or equivalently,
\begin{align*}
 B_2\sqrt{\frac{A_3}{A_2}}e^{-\lambda\tau/2}\leq u\leq \frac{1}{2}.
\end{align*}
For $1/2\leq u < 1$, writing $v:=1-u^2$, then in view of \eqref{eq:est2}, we have
\begin{align*}
 e^{2\lambda\tau}\mathcal{D}_u\left[z_{\text{ext}}^+\right] & \geq C_{10} \left(A_2^2-C_7 A_2 A_3 e^{-\lambda\tau}v^{\lambda-1} - C_8 A_3^2e^{-2\lambda\tau}v^{2\lambda-2} \right)v^{2\lambda}\\
& \geq 0
\end{align*}
if $\tau\geq\tau_2$ with $\tau_2$ sufficiently large.

Therefore, $z_{\text{ext}}^+$ is indeed a supersolution. That $z_{\text{ext}}^-$ is a subsolution is proved similarly.
\end{proof}


\section{Upper and lower barriers}
Recall the formal solution $z_{\text{form}}$ constructed in Section \ref{formal}. We say that a function $z^{-}$ ($z^{+}$) is a lower (upper) barrier to equation  \eqref{eq:RFzu} if $z^{-}$ is a subsolution (supersolution) that lies below (above) $z_{\text{form}}$ in an appropriate space-time region. The main result of this section is the following.

\begin{prop}\label{barriers}
There exist a sufficiently large $\tau_0<\infty$ and positive piecewise smooth functions $z^{\pm} = z^{\pm}(u,\tau)$ defined for $0<u<1$ and $\tau\geq \tau_0$, such that the following are true.
\begin{enumerate}
 \item[(B1)] $z^{\pm}$ are upper ($+$) and lower ($-$) barriers to equation \eqref{eq:RFzu} for $0<u<1$ and $\tau\geq \tau_0$.
 \item[(B2)] Near $u=0$, $z^{\pm}=z^{\pm}_{\text{int}}$; near $u=1$, $z^{\pm}=z^{\pm}_{\text{ext}}$.
 \item[(B3)] At any $\tau\in[\tau_0,\infty)$, $\lim\limits_{u\searrow 0} z^{-} = \lim\limits_{u\searrow 0} z^{+} =1$, and $\lim\limits_{u\nearrow 1}z^{-}=\lim\limits_{u\nearrow 1}z^{+}=0$.
 \item[(B4)] At any $\tau\in[\tau_0,\infty)$, there exists a constant $K$ independent of $\tau$ such that
\begin{align}
 \left\vert z^{\pm}_u/u\right\vert,\;\left\vert z^{\pm}_{uu}\right\vert\leq K e^{\lambda\tau}
\end{align}
at points where $z^{\pm}$ are smooth.
\end{enumerate}
\end{prop}

The proposition will follow from several lemmata. We first explain the idea behind its proof. We properly order $z^{\pm}_{\text{ext}}$ and $z^{\pm}_{\text{int}}$ so that $z^{-}_{\text{int}}\leq z^{+}_{\text{int}}$ and $z^{-}_{\text{ext}}\leq z^{+}_{\text{ext}}$. We then patch together $z_{\text{int}}^{+}$ and $z_{\text{ext}}^{+}$ near the interior-exterior interface to obtain an upper barrier. A similar patching argument yields a lower barrier.

\begin{lemma}\label{int_bar}
Let $\beta$ be defined as in Lemma \ref{z_int}. Let $A_1^+$ and $A_1^-$ denote the constant $A_1$ in $z_{\text{int}}^+$ and $z_{\text{int}}^-$, respectively. For $A_1^->A_1^+$, there exists $\tau_3\geq\tau_1$ such that
\begin{align*}
 z^{\pm}_{\text{int}} = \mathfrak{B}(A_1^{\pm}r) \pm e^{-\lambda\tau}\beta
\end{align*}
are properly ordered so that $z^{-}_{\text{int}}\leq z^{+}_{\text{int}}$ in $\Omega_{\text{int}}$ for all $\tau\geq\tau_3$.
\end{lemma}
\begin{proof}
For $A_1^- > A_1^+$, using the asymptotic expansions of $\mathfrak{B}$ and $\beta$ (cf. the proof of Lemma \ref{z_int}) we have the following. Near $r=0$, with $b_2>0$,
\begin{align*}
 z_{\text{int}}^{+} - z_{\text{int}}^{-} & = \left\{ b_2\left[(A_1^-)^2 - (A_1^+)^2 \right] + 2\left[1+o(1)\right]e^{-\lambda\tau} \right\} r^2 + O\left(r^4\right)\\
& > 0 \quad\text{as }r\searrow 0.
\end{align*}
Near $r=\infty$, with $\hat{A}=1+\frac{\lambda+1}{2}$,
\begin{align*}
 z_{\text{int}}^{+} - z_{\text{int}}^{-} & = \left[(A_1^+)^{-2}-(A_1^-)^{-2}\right]\left\{r^{-2}-2 \left[\hat{A}+o(1)\right]e^{-\tau}\right\} + O\left(r^{-4}\right)\\
& > 0
\end{align*}
for sufficiently large $\tau$ and $r$. On any bounded interval $c<r<C$ and for sufficiently large $\tau$, it is straightforward to check that $z^{+}_{\text{int}}> z^{-}_{\text{int}}$. Thus, the lemma follows.
\end{proof}

\begin{lemma}\label{ext_bar}
Let $Z_1$ and $\zeta$ be defined as in Lemma \ref{z_ext}. Let $A_2^+$ and $A_2^-$ denote the constant $A_2$ in $z_{\text{ext}}^+$ and $z_{\text{ext}}^-$, respectively. For $A_2^+>A_2^-$, if we relabel $A_3:=\max\left\{A_3(A_2^+),A_3(A_2^-)\right\}$, $B_2:=\max\left\{B_2(A_2^+),B_2(A_2^-)\right\}$, and $\tau_2:=\max\left\{\tau_2(A_2^+),\tau_2(A_2^-)\right\}$, then there exists $\tau_4\geq\tau_2$ such that
\begin{align*}
 z_{\text{ext}}^{\pm}(u,\tau) = e^{-\lambda\tau}A^{\pm}_2 Z_1(u) \pm e^{-2\lambda\tau} A_3 \zeta(u)
\end{align*}
are properly ordered $z^{-}_{\text{ext}}\leq z^{+}_{\text{ext}}$ in $\Omega_{\text{ext}}$\footnote{In definition \eqref{eq:def_ext} of $\Omega_{\text{ext}}$, we replace $A_2$ with $A_2^{-}$ since $A_2^+>A_2^-$.} for all $\tau\geq\tau_4$.
\end{lemma}
\begin{proof}
For $A_2^+ > A_2^-$, the asymptotic expansions \eqref{eq:zeta_asymp} of $\zeta$ imply the following. As $u\searrow 0$, $z_{\text{ext}}^{+} > z_{\text{ext}}^{-}$. As $u\nearrow 1$,
\begin{align*}
 z_{\text{ext}}^{+} - z_{\text{ext}}^{-} & = e^{-\lambda\tau}\left(1-u^2\right)^{1+\lambda}\left\{ (A_2^+ - A_2^-)u^{-2} - 2A_3e^{-\lambda\tau}\left(1-u^2\right)^{\lambda-1} \right\} \\
&\quad + e^{-2\lambda\tau} O\left(\left(1-u^2\right)^\lambda\log\left(1-u^2\right)\right)\\
& > 0
\end{align*}
for all $\tau$ sufficiently large. On any interval $0<a\leq u\leq b<1$ and for sufficiently large $\tau$, $z_{\text{ext}}^{+} > z_{\text{ext}}^{-}$ by a direct computation. Thus, the lemma is proved.
\end{proof}

For sufficiently large $\tau$, the regions $\Omega_{\text{int}}$ and $\Omega_{\text{ext}}$ intersect. In below, we state and prove a patching lemma for $z_{\text{int}}^{+}$ and $z_{\text{ext}}^{+}$. We omit the patching lemma for $z_{\text{int}}^{-}$ and $z_{\text{ext}}^{-}$, since its statement and proof are entirely analogous. To shorten the notation, we write $A_1^+$, $A_2^+$ as $A_1$, $A_2$.

\begin{lemma}\label{patch}
Let $R_D := D\sqrt{A_3/A_2}$ where $D>0$ is arbitrary. Suppose $A_1$ and $A_2$ satisfy the following inequality
\begin{align}\label{eq:a1a2}
 \left( 1 + \frac{3}{8}D^{-2} \right) A_2 < A_1^{-2} < \left( 1 + \frac{1}{2}D^{-2} \right) A_2. 
\end{align}
Then there exists $\tau_5:=\max\left\{\tau_3,\tau_4\right\}$ sufficiently large such that
\begin{align}
z_{\text{int}}^+ & \leq z_{\text{ext}}^+ \quad\text{at } r=R_D, \label{eq:patch1}\\
z_{\text{int}}^+ & \geq z_{\text{ext}}^+ \quad\text{at } r=2R_D, \label{eq:patch2}
\end{align}
for $\tau\geq\tau_5$.
\end{lemma}
\begin{proof}
At the interface of the interior and exterior regions, we have the following when $\tau\geq\tau_5$. From the interior region, as $r\nearrow\infty$, $\mathfrak{B}(r)=r^{-2}+c_2 r^{-4}+O\left(r^{-6}\right)$, and so
\begin{align*}
z_{\text{int}}^+ = A_1^{-2}r^{-2} + c_2 A_1^{-4}r^{-4}+O\left(r^{-6}\right)+O\left(e^{-\lambda\tau}\right)\quad\text{as }r\nearrow\infty.
\end{align*}
From the exterior region, as $u\searrow 0$, using $u=re^{-\lambda\tau/2}$ and \eqref{eq:zeta_asymp}, we have on any compact $r$-interval,
\begin{align*}
z_{\text{ext}}^+ & = A_2 e^{-\lambda\tau} u^{-2}\left(1-u^2\right)^{\lambda+1} + A_3 e^{-2\lambda\tau} u^{-4}\left(1+O\left(u^2\log u\right)\right)\\
& = A_2 r^{-2} + A_3 r^{-4} + O\left(\tau e^{-\lambda\tau}\right).
\end{align*}
Then on bounded $r$-interval, one has
\begin{align*}
r^2\left(z_{\text{int}}^+ -z_{\text{ext}}^+ \right) & = \left(A_1^{-2} - A_2\right) + \left( c_2A_1^{-4}+O\left(r^{-2}\right)-A_3 \right)r^{-2} + O\left(\tau e^{-\lambda\tau}\right).
\end{align*}
We can choose a constant $\hat C$ so large that for
\begin{align*}
A_3 \geq \hat{C} A_1^{-4} \quad\text{and}\quad A_3 \geq \hat{C}\sqrt{A_2},
\end{align*}
we have
\begin{align*}
\left\vert \frac{c_2A_2}{A_3A_1^{4}} + O\left(\frac{A_2^2}{A_3^2} \right) \right\vert \leq \frac{A_2}{2}.
\end{align*}
Then at $r=R_D$,
\begin{align*}
R_D^2\left(z_{\text{int}}^+ -z_{\text{ext}}^+ \right) & = \left(A_1^{-2} - A_2\right) + \left[ \frac{c_2A_2}{A_3 A_1^4} + O\left(\frac{A_2^2}{A_3^2}\right)-A_2 \right]D^{-2} + O\left(\tau e^{-\lambda\tau}\right)\\
& \leq A_1^{-2} - \left(1+\frac{1}{2}D^{-2}\right) A_2 + O\left(\tau e^{-\lambda\tau}\right),
\end{align*}
and at $r=2R_D$,
\begin{align*}
4R_D^2\left(z_{\text{int}}^+ -z_{\text{ext}}^+ \right) & = \left(A_1^{-2} - A_2\right) + \left[ \frac{c_2A_2}{A_3A_1^4} + O\left(\frac{A_2^2}{A_3^2}\right)-A_2 \right]\frac{D^{-2}}{4} + O\left(\tau e^{-\lambda\tau}\right)\\
& \geq A_1^{-2} - \left(1+\frac{3}{8}D^{-2}\right) A_2 + O\left(\tau e^{-\lambda\tau}\right).
\end{align*}
Now choose $A_1$ and $A_2$ according to \eqref{eq:a1a2}; then the lemma follows for $\tau\geq\tau_5$.
\end{proof}

Lemmata \ref{int_bar}, \ref{ext_bar}, and \ref{patch} allow us to construct barriers for equation \eqref{eq:RFzu}. From now on, we define the function $z^{+}=z^{+}(u,\tau)$ by
\begin{align}\label{eq:zplus}
 z^{+} := \left\{
\begin{array}{ccc}
 z^+_{\text{int}},& \text{if}& 0<u\leq e^{-\lambda\tau/2}R_D , \vspace{5pt} \\ 
 \min\left\{z^+_{\text{int}}, z^+_{\text{ext}}\right\},& \text{if}& e^{-\lambda\tau/2}R_D\leq u\leq 2e^{-\lambda\tau/2}R_D,  \vspace{5pt} \\
 z^+_{\text{ext}},& \text{if}& 2e^{-\lambda\tau/2}R_D\leq u<1,\\
\end{array}
\right.
\end{align}
for $\tau\geq\tau_5$.
We define $z^{-}=z^{-}(u,\tau)$ analogously using $z^{-}_{\text{int}}$ and $z^-_{\text{ext}}$. In particular, for $e^{-\lambda\tau/2}R_D\leq u\leq 2e^{-\lambda\tau/2}R_D$, $z^{-}:=\max\left\{z^-_{\text{int}}, z^-_{\text{ext}}\right\}$.

\begin{lemma}\label{deriv_bdd}
Let $\tau\geq \tau_5$. There exists a constant $K$ independent of $\tau$ such that
\begin{align*}
 \left\vert z^{\pm}_u/u\right\vert,\;\left\vert z^{\pm}_{uu}\right\vert\leq K e^{\lambda\tau}
\end{align*}
at points where $z^{\pm}$ are smooth.
\end{lemma}
\begin{proof}
At a point where $z^{+}$ is smooth, $z^{+}$ is either $z^{+}_{\text{int}}$ or $z^{+}_{\text{ext}}$.

Suppose $z^{+}$ is smooth at $u\in (0, 2e^{-\lambda\tau/2} R_D)$ and $z^{+}=z^{+}_{\text{int}}$, then
\begin{align*}
z^+ & = \mathfrak{B}\left(A_1 r\right) + e^{-\lambda\tau}\beta\left( r \right)\\
& = 1 + C_1 r^2 + o\left( r^2 \right) + e^{-\lambda\tau}\left(r^2+o\left( r^2 \right)\right)\quad \text{as } r\searrow 0,\\
& = 1 + C_1 e^{\lambda\tau} u^2 + e^{\lambda\tau} o\left( u^2 \right) + u^2+o\left( u^2 \right)\quad \text{as } u\searrow 0.
\end{align*}
So then
\begin{align*}
z^+_{u} & = e^{\lambda\tau}\left(C_2 u + o \left(u\right)\right) + u+o\left(u\right) \quad \text{as } u\searrow 0,\\
z^+_{uu} & = e^{\lambda\tau}\left(C_3 + o\left(1\right)\right) + 1+o\left(1\right) \quad \text{as } u\searrow 0.
\end{align*}
Thus, there exists a constant $K_1$ such that for $0< u < 2 e^{-\lambda\tau/2} R_D$,
\begin{align}\label{eq:k1}
 \left\vert z^+_{u}/u \right\vert,\; \left\vert z^+_{uu} \right\vert\leq K_1 e^{\lambda\tau}.
\end{align}

Suppose $z^{+}$ is smooth at $u\in \left(e^{-\lambda\tau/2} R_D, 1 \right)$ and $z^{+}=z^{+}_{\text{ext}}$, then
\begin{align*}
 z^+ & = e^{-\lambda\tau}A_2 Z_1(u) +  e^{-2\lambda\tau} A_3 \zeta(u),
\end{align*}
where $Z_1(u) = u^{-2}\left(1-u^2\right)^{\lambda+1}$ for $\lambda\geq1$, and $\zeta\left(u\right)$ is a smooth solution to the ODE \eqref{eq:ODE_zeta}. So then
\begin{align*}
 \left\vert z^{+}_u/u \right\vert &\lesssim e^{-\lambda\tau}\left\vert Z_1'/u \right\vert + e^{-2\lambda\tau}\left\vert\zeta'/u\right\vert,\\
 \left\vert z^{+}_{uu} \right\vert &\lesssim e^{-\lambda\tau} \left\vert Z_1''\right\vert+e^{-2\lambda\tau}\left\vert \zeta''\right\vert.
\end{align*}
From the definition of $Z_1$, we compute
\begin{align}
 Z_1'/u & = -2\left( u^{-4} + \lambda u^{-2} \right)\left(1-u^2\right)^{\lambda}, \label{eq:Z1'}\\
 Z_1'' & = 2 \left(1-u^2\right)^{\lambda-1} \left[ 3u^{-4} + 3\left(\lambda-1\right)u^{-2} + \left(2\lambda-1\right)\lambda \right]. \label{eq:Z1''}
\end{align}
From equation \eqref{eq:ODE_zeta}, we have
\begin{align*}
 -\frac{1}{2}\zeta' = \frac{\left(1-u^2\right)^{2\lambda-1}}{u^5} + \frac{u^{-1} + \lambda u}{\left(1-u^2\right)}\zeta.
\end{align*}
Then using \eqref{eq:zeta_asymp} we obtain, writing $v:=1-u^2$,
\begin{align}\label{eq:zeta'}
 \left\vert \zeta'/u \right\vert \lesssim \left\{
\begin{array}{cc}
u^{-6} + O\left(u^{-4}\log u\right)& \text{as } u\searrow 0, \\
v^{2\lambda-1} + O\left(v^{2\lambda}\log v \right)& \text{as } u\nearrow 1,
\end{array}\right.
\end{align}
and similarly,
\begin{align}\label{eq:zeta''}
 \left\vert \zeta'' \right\vert \lesssim \left\{
\begin{array}{cc}
u^{-6} + O\left(u^{-4}\log u\right)& \text{as } u\searrow 0, \\
v^{2(\lambda-1)} + O\left( v^{2\lambda-1}\log v\right)& \text{as } u\nearrow 1.
\end{array}\right.
\end{align}
Thus, by \eqref{eq:Z1'}--\eqref{eq:zeta''}, there exist constants $K_2, K_3$ such that for $e^{-\lambda\tau/2} R_D<u<1$,
\begin{align*}
 \left\vert z^{+}_u/u \right\vert \leq K_2 e^{\lambda\tau},\quad \left\vert z^{+}_{uu} \right\vert \leq K_3 e^{\lambda\tau}.
\end{align*}

Choose $K=\max\left\{K_1,K_2,K_3\right\}$, then the lemma is true for $z^{+}$. The proof for $z^{-}$ is similar.
\end{proof}

We can now prove Proposition \ref{barriers}.
\begin{proof}[Proof of Proposition \ref{barriers}]
Since $\lim\limits_{u\searrow 0} z^{\pm}_{\text{int}}=1$, $z^{\pm}_{\text{int}}>0$ on $0< r\leq 2R_D$ for sufficiently small $D$. Since $Z_1(u)\geq0$, there exists a sufficiently large $\tau_0\geq\tau_5$ such that $z^{\pm}_{\text{ext}}>0$ on $e^{-\lambda\tau/2}R_D\leq u <1$. Thus, $z^{\pm}$ are positive piecewise smooth functions for $0<u<1$ and $\tau\geq\tau_0$. The minimum (maximum) of two supersolutions (subsolutions) is still a supersolution (subsolution), so (B1) is true. One verifies (B2) and (B3) directly using the definition of $z^{\pm}$ and the properties of $z^{\pm}_{\text{int}}$ and $z^{\pm}_{\text{ext}}$. (B4) follows from Lemma \ref{deriv_bdd}.
\end{proof}


\section{Existence and uniqueness of complete solutions}
We first prove a comparison principle for equation \eqref{eq:RFzu}. Similar results have appeared in \cites{ACK12, M12}.
\begin{lemma}\label{compare}
Let $\bar\tau\in[\tau_0,\infty)$ be arbitrary. Let $z^{\pm}$ be two nonnegative sub-($-$) and super- ($+$) solutions of equation \eqref{eq:RFzu} respectively. Suppose there exists a constant $K$ such that either $\left\vert z_{u}^{-}/u \right\vert$ and $\left\vert z_{uu}^{-} \right\vert$, or $\left\vert z_{u}^{+}/u\right\vert$ and $\left\vert z_{uu}^{+} \right\vert$, are bounded by $K e^{\lambda\tau}$. Moreover, assume
\begin{enumerate}
 \item[(C1)] $z^-\left(u,\tau_0\right) < z^{+}\left(u,\tau_0\right)$ for $0 < u < 1$;
 \item[(C2)] $z^{-}\left(0,\tau\right) \leq z^{+}\left(0,\tau\right)$, and $z^{-}\left(1,\tau\right) \leq z^{+}\left(1,\tau\right)$ for all $\tau\in[\tau_0,\bar\tau]$.
\end{enumerate}
Then $z^{-}\left(u,\tau\right)\leq z^{+}\left(u,\tau\right)$ in $[0,1]\times[\tau_0,\bar\tau]$.
\end{lemma}
\begin{remark}
In this lemma, we assume $z^{\pm}$ are smooth. The result also holds for piecewise smooth $z^{\pm}$. When applying the comparison principle, we will only evaluate $z^{\pm}$ at ``points of first contact with a given smooth function'' which are necessarily smooth points of $z^{\pm}$ for each $\tau\geq\tau_0$.
\end{remark}
\begin{proof}[Proof of Lemma \ref{compare}]
Suppose $\left\vert z^{+}_u/u \right\vert, \left\vert z^{+}_{uu}\right\vert \leq K e^{\lambda\tau}$. For $\mu>0$ to be chosen and arbitrary $\varepsilon>0$, define a function
\begin{align*}
 w := e^{-\mu e^{\lambda\tau}}\left(z^+ - z^- \right) + \varepsilon.
\end{align*}
Then $w>0$ on the parabolic boundary of the evolution by assumptions (C1) and (C2). We claim that $w>0$ in $(0,1)\times[\tau_0,\bar\tau]$. Suppose the contrary, then there must be an interior point $u_\ast$ and a first time $\tau_\ast$ such that $w(u_\ast,\tau_\ast)=0$ and $w_\tau(u_\ast,\tau_\ast)\leq 0$. Moreover, at $(u_\ast,\tau_\ast)$, we have
\begin{align*}
 z^{+} = z^{-} - \varepsilon e^{-\mu e^{\lambda\tau_\ast}},\quad z^{+}_u = z^{-}_u, \quad z^{+}_{uu}\geq z^{-}_{uu}.
\end{align*}
Then at $(u_\ast,\tau_\ast)$, 
\begin{align*}
0  & \geq  e^{\mu e^{\lambda\tau_\ast}} w_\tau \\
& = \left(z^{+}_\tau- z^{-}_\tau \right) - \lambda\mu e^{\lambda\tau_\ast}\left(z^+ - z^-\right)\\
& \geq \left(z^+ - z^- \right)\left(u^{-2}-\lambda\mu e^{\lambda\tau_\ast}\right) + \frac{\mathcal{Q}_u[z^+]-\mathcal{Q}_u[z^-]}{2(n-1)}\\
& = \left(z^- - z^+\right)\left\{\lambda\mu e^{\lambda\tau_\ast} + \frac{\left(z^+_u/u \right) -z^+_{uu} }{2(n-1)} + \frac{z^+ + z^- - 1}{u^2} \right\} + z^-\left(z^+_{uu} - z^-_{uu} \right)\\
& \geq  \varepsilon e^{-\mu e^{\lambda\tau_\ast}}\left\{\lambda\mu e^{\lambda\tau_\ast} - \frac{K e^{\lambda\tau_\ast}}{(n-1)} - \frac{1}{u^2_\ast} \right\}.
\end{align*}
Choose $\mu$ so large that $\lambda\mu > K/(n-1)+u^{-2}_\ast e^{-\lambda\tau_\ast}$. Then at $(u_\ast,\tau_\ast)$,
\begin{align*}
 0  \geq w_\tau > 0,
\end{align*}
which is a contradiction. This proves the lemma in the case $\left\vert z^{+}_u/u\right\vert, \left\vert z^{+}_{uu}\right\vert\leq K e^{\tau}$.

The case when $\left\vert z^{-}_u/u \right\vert, \left\vert z^{-}_{uu} \right\vert\leq K e^{\lambda\tau}$ is proved analogously because at the interior first contact point $\left(u_\ast,\tau_\ast\right)$, we have
\begin{align*}
 e^{\mu e^{\lambda\tau_\ast}} w_\tau & = \left(z^{-} - z^{+}\right)\left\{\mu\lambda e^{\lambda\tau_\ast} + \frac{\left(z^-_u/u\right) -z^{-}_{uu} }{2(n-1)} + \frac{z^+ + z^{-} - 1}{u^2} \right\} \quad + z^+\left(z^+_{uu} - z^-_{uu} \right).
\end{align*}

Therefore, the lemma is proved.
\end{proof}

Now for any solution $z$ of equation \eqref{eq:RFzu} we have the following.
\begin{lemma}\label{complete}
Suppose $0< z\leq z^{+}$. If $\lambda\geq 1$, then $z$ determines a \emph{complete} rotationally symmetric metric $g:=z^{-1} d\psi^2+\psi^2 g_{\text{sph}}$ on $\RR^N$.
\end{lemma}
\begin{proof}
By definition $g$ is rotationally symmetric. To see that $g$ is a complete metric, it suffices to show that any radial geodesic $\eta$ starting from the origin has infinite length in the $s$-coordinate. The length of $\eta$ in $s$-coordinate is a function of $u$ and $\tau$ given by
\begin{align*}
 s(u,\tau) = e^{-\tau/2} \sigma(u) = e^{-\tau/2} \int_{0}^u \frac{d\sigma}{d \hat u} d \hat u.
\end{align*}
Since $z = \psi_s^2 = 2\left(n-1\right) u_\sigma^2$, and $0< z \leq z^{+}$ by hypothesis, we have
\begin{align*}
\frac{\sigma(u)}{\sqrt{2(n-1)}} & \geq \int_{u_0}^u \frac{1}{\sqrt{z}}d \hat u  \geq  \int_{u_0}^u \frac{1}{\sqrt{z^{+}}}d \hat u.
\end{align*}
As $u\nearrow 1$,
\begin{align*}
z^{+}_{\text{ext}} & = e^{-\lambda\tau} A_2 u^{-2}\left(1-u^2\right)^{1+\lambda} + e^{-2\lambda\tau} A_3 \zeta(u)\\
&  =  e^{-\lambda \tau}A_2 u^{-2}\left(1-u^2\right)^{\lambda+1} + e^{-2\lambda\tau}A_3\left\{-\left(1-u^2\right)^{2\lambda} + O\left(\left(1-u^2\right)^{2\lambda+1}\log\left(1-u^2\right)\right)\right\}.
\end{align*}
So for $u_0$ near $1$ and $\tau_0$ sufficiently large, $z^{+}=z^{+}_{\text{ext}}$ in $[u_0,1)\times[\tau_0,\infty)$ with
\begin{align*}
z^{+}_{\text{ext}} & \leq e^{-\lambda\tau} u^{-2}\left(1-u^2\right)^{1+\lambda} \left(\frac{3A_2}{2}\right).
\end{align*}
It follows that
\begin{align*}
\frac{s(u,\tau)}{\sqrt{2(n-1)}} & \geq e^{-\tau/2} \int_{u_0}^u \frac{1}{\sqrt{z^{+}}}d \hat u\\
& = e^{-\tau/2} \int_{u_0}^u \frac{1}{\sqrt{z^{+}_{\text{ext}}}}d \hat u\\
& \geq \sqrt{\frac{3A_2}{2}} e^{\lambda\tau/2} \int_{u_0}^u \frac{\hat u}{\left(1-\hat{u}^2\right)^{(1+\lambda)/2}} d\hat u.
\end{align*}
Hence,
\begin{align*}
\frac{s(u,\tau)}{\sqrt{3(n-1)A_2}} & \geq \left\{
\begin{array}{cc}
\log\left(1-u_0^2\right) - \log\left(1-u^2\right), & \lambda=1, \vspace{5pt} \\
\frac{e^{(\lambda-1)\tau/2}}{(\lambda-1)}\left\{\left(1-u^2\right)^{\left(1-\lambda\right)/2} -\left(1-u_0^2\right)^{(1-\lambda)/2}\right\}, &\lambda>1.
\end{array}
\right.
\end{align*}
Therefore, for each $\tau\geq\tau_0$, $\lim\limits_{u\nearrow 1} s(u,\tau) = \infty$, and so the lemma is proved.
\end{proof}

We are now ready to prove our main result.
\begin{proof}[Proof of Theorem \ref{thm1}]
Let $N\geq 3$, $\lambda\geq 1$, and $A_1^{-}, A_1^{+}>0$ with $A_1^{-}>A_1^{+}$. Choose $A_2^{\pm}$ according to \eqref{eq:a1a2}. Let $\hat{z}_0$ be the function obtained by patching together $\mathfrak{B}(\tilde A_1r)$ and $\tilde{A}_2 Z_1(u)$. Because $z^{-}(u,\tau_0)<z^{+}(u,\tau_0)$, we can smooth out $\hat{z}_0$ to obtain a smooth initial profile function $z_0$ with $0<z^{-}(u,\tau_0)< z_0 < z^{+}(u,\tau_0)$ for $0<u<1$. By Lemma \ref{complete}, $z_0$ determines a complete rotationally symmetric metric $g_0$ on $\RR^N$.

It is straightforward to check that $g_0$ has bounded sectional curvatures everywhere. In particular, the sectional curvatures $K_0, L_0$ of $g_0$ at the origin $O$ can be computed by
\begin{align*}
\left. K_0 \right\vert_{O} = \left. L_0 \right\vert_{O} = \lim\limits_{x\searrow 0} \frac{1-\psi^2_s}{\psi^2}=\lim\limits_{r\searrow 0}\frac{1-z_0}{r^2}e^{(\lambda+1)\tau_0}.
\end{align*}
So using the barriers $z^{\pm}(u,\tau_0)$ and that $e^{(\lambda+1)\tau_0} = (T-t_0)^{-(\lambda+1)}$, we can find constants $C^{\pm} = C(N,A_1^{\pm})$ such that
\begin{align}\label{eq:C-pm}
 \frac{C^{+}}{(T-t_0)^{\lambda+1}} \leq \left. K_0 \right\vert_{O} = \left. L_0 \right\vert_{O} \leq  \frac{C^{-}}{(T-t_0)^{\lambda+1}}. 
\end{align}
Since the sectional curvatures depend smoothly on the metric, there is a neighborhood $\CMcal{G}_N$ of $g_0$ in $C^2$ topology that corresponds to an open set of $z$'s around $z_0$, all of which lie between $z^{-}(u,\tau_0)$ and $z^{+}(u,\tau_0)$ and determine complete rotationally symmetric metrics on $\RR^N$ for which \eqref{eq:C-pm} is satisfied.

Now let $g_0\in\CMcal{G}_N$ be arbitrary. There exists a unique solution $g(t)$ to Ricci flow for $t\in[0, T_0)$ with $g(0)=g_0$ \cites{Shi89-1, ChZh06}. By expression \eqref{eq:g(s)}, $g_0$ corresponds to a function $\psi$ with $\psi(s,0)<r_0$ for some constant $r_0>0$. Since the metric $\tilde g_{t} = ds^2 + \tilde\psi(t)^2 g_{\text{sph}}$ with $\tilde\psi(0)\equiv r_0$ is the shrinking cylinder solution to Ricci flow on $\RR\times S^n$, $\psi(t)\leq \tilde\psi(t)$ where $\tilde\psi(t)\searrow 0$ in finite time. So $g(t)$ encounters a global singularity.

The profile $z(u,\tau)$ of $g(t)$ is the unique solution of equation \eqref{eq:RFzu} for $0<u<1$ and $\tau\geq\tau_0$, with boundary conditions $z(0,\tau) = 1$ and $z(1,\tau)=0$, and initial condition $z(u,\tau_0) = z_0$. The barriers $z^{\pm}$ satisfy the hypotheses of Lemma \ref{compare}, so by the comparison principle $z^{-}\leq z(u,\tau)\leq z^{+}$ for $\tau_0\leq\tau<\infty$. So for $0\leq t < T=e^{-\tau_0}$, the metric $g(t)$ corresponding to $z(u,\tau)$ is a complete metric on $\RR^N$ by Lemma \ref{complete}. Moreover, one checks using the barriers that the sectional curvatures $K(t), L(t)$ of $g(t)$ at the origin $O$ satisfy
\begin{align*}
\frac{C^{+}}{(T-t)^{\lambda+1}} \leq \left. K\left(t\right) \right\vert_{O} = \left. L\left(t\right) \right\vert_{O} \leq  \frac{C^{-}}{(T-t)^{\lambda+1}}.
\end{align*}
So part (1) of Theorem \ref{thm1} is proved.

Since $z^{-}\leq z(u,\tau)\leq z^{+}$ for any $\tau<\infty$, the solution $z(u,\tau)$ exhibits the asymptotic behavior of $z^{\pm}$. Near the origin, $z(u,\tau)$ converges uniformly to the Bryant soliton profile function for $0<u<R_D e^{-\lambda\tau}$ and $\tau\nearrow\infty$. Near spatial infinity, $u\nearrow 1$ while $z(u,\tau)\searrow 0$. Thus, $g(t)$ has the asymptotic behavior described in parts (2) and (3) of Theorem \ref{thm1}. Lastly, the asymptotic behavior of the solution agrees with that of the barriers, and hence with that of the formal solution, so part (4) of the theorem is true.
\end{proof}


\section{Relation to the standard solutions}\label{stansolsect}
In \cite{P03-1}, Perelman described a special family of Ricci flow solutions, the so-called standard solutions, on $\RR^3$. These solutions are complete rotationally symmetric with nonnegative sectional curvature, and split at infinity as the metric product of a ray and the round $S^2$ of constant scalar curvature.

Consider a rotationally symmetric metric $g_0$ on $\RR^N$ ($N\geq 3$) with the following properties:
\begin{enumerate}
\item[(P1)] $\Rm_{g_0}\geq 0$ everywhere with $\Rm_{g_0}>0$ at some point.
\item[(P2)] The curvature $\left\vert \Rm_{g_0}\right\vert$ and its derivatives $\left\vert \nabla^i \Rm_{g_0}\right\vert$, $i=1,2,3,4$, are bounded.
\item[(P3)] There is a sequence of points $y_k\to\infty$ in $\RR^N$ such that $\left(\RR^N, g_0, y_k\right)$ converges in $C^3$ pointed Cheeger-Gromov topology to $\RR\times S^n(r_0)$ for some constant $r_0>0$.
\end{enumerate}
Following \cite{LT}, a Ricci flow solution $g(t)$ whose initial condition satisfies (P1)--(P3) is called a standard solution. A standard solution of Ricci flow is unique up to the first singular time \cites{LT, ChZh06}.

\begin{lemma}\label{stansol}
Let $\CMcal{G}_N$ be given as in Theorem \ref{thm1}. There is an open (in $C^6$ topology) set $\CMcal{G}^\ast_N\subset \CMcal{G}_N$ of metrics that satisfy properties (P1)--(P3).
\end{lemma}
\begin{proof}
Define 
\begin{align*}
\tilde{\CMcal{G}}^\ast_N:=\left\{ g_0\in\CMcal{G}_N: g_0\text{ satisfies P(1)--P(3)}\right\}.
\end{align*}
We first show that $\tilde{\CMcal{G}}^\ast_N$ is nonempty.

Let $\tau=\tau_0$ correspond to $t=0$. By the proof of Theorem \ref{thm1}, there exists $\hat{z}_0$ which is obtained by patching scaled copies of $\mathfrak{B}$ and $Z_1$. Let $\hat{g}_0$ be the metric determined by the profile function $\hat{z}_0$. For $\hat{g}_0$,  $K=-(z_u/2u)e^{\tau_0}=-(z_r/2r)e^{(\lambda+1)\tau_0}>0$ at the origin. Observe that the patching occurs in $R_D\leq r\leq 2R_D$, where $R_D := D\sqrt{A_3/A_2}$ for an arbitrary constant $D>0$. So by the continuity of $K$ there exists $D_0$ such that $K>0$ for $0< r\leq 2R_0$, where $R_0:=R_{D_0}$. On the other hand, where $\hat{z}_0=A_2u^{-2}\left(1-u^2\right)^{1+\lambda}$, we have 
\begin{align*}
 K & = - \frac{z_{\psi}}{2\psi} = - \frac{z_u}{2u} e^{\tau} = A_2 u^{-4} \left(1-u^2\right)^\lambda \left(1+\lambda u^2\right)e^{\tau}>0.
\end{align*}
Hence, the piecewise smooth function $\hat{z}_0$ determines a metric $\hat{g}_0$ for which $K>0$ on $\RR^N$ where $\hat{g}_0$ is smooth; moreover, $K\searrow 0$ as $u\nearrow 1$, i.e., as one approaches spatial inifinity. Since $z^{-}<z^{+}$, we can smooth $\hat{z}_0$ to obtain a smooth metric $g_0$ for which $K\geq0$ everywhere with $K>0$ at the origin, and $g_0\in\CMcal{G}_N$. Also for this metric $g_0$, because $L=(1-z)/\psi^2$, $L\geq 0$ everywhere with $L>0$ at the origin, and $L\to 1/\psi^2$ as we approach spatial infinity. Thus, $g_0$ satisfies (P1).

To check (P2), we first note that $\left\vert \Rm_{g_0} \right\vert$ is bounded by the proof of Theorem \ref{thm1}. The derivatives $\nabla^i\Rm_{g_0}$ ($i\in\NN$) are determined by $\p^i_s K$ and $\p^i_s L$. Recall that $s(u,\tau)=e^{-\tau/2}\sigma(u)$ and $z=2(n-1)u_\sigma^2$. Then at time $\tau_0$,
\begin{align*}
\frac{\p s}{\p u} = \frac{\p \sigma}{\p u}e^{-\tau_0/2} = e^{-\tau_0/2}\frac{\sqrt{2(n-1)}}{\sqrt{z_0}}.
\end{align*}
Since $0<z^{-}<z_0<z^{+}$, arguing as in the proof of Lemma \ref{complete}, there exists $u_0\in(0,1)$ such that for $u_0\leq u <1$,
\begin{align}\label{eq:u_s}
\frac{\p u}{\p s}& \lesssim \sqrt{z^{+}_{\text{ext}}} \lesssim \left(1-u^2\right)^{(\lambda+1)/2}\quad (\lambda\geq 1).
\end{align}
By the chain rule that $\p_s = \left(\p u/\p s\right) \p_u$, one checks that
\begin{align*}
 \left\vert K_s \right\vert \lesssim \left(1-u^2\right)^{(3\lambda-1)/2},\quad \left\vert L_s \right\vert \lesssim \left(1-u^2\right)^{(\lambda+1)/2}+O\left(\left(1-u^2\right)^{(3\lambda+1)/2}\right).
\end{align*}
So $K_s$ and $L_s$ are bounded. Similarly, direct computation shows that $\left\vert \p^i_s K \right\vert$ and $\left\vert \p^i_s L \right\vert$ are bounded for $i=2,3,4$. If $0<u\leq u_0$, then we are looking at a compact subset of $\RR^N$ where $\left\vert \nabla^i\Rm_{g_0} \right\vert$ are bounded for any $i\in\NN$ because $g_0$ is smooth. Thus, $g_0$ satisfies (P2).

To check (P3), we let $\{y_k\}_{k=0}^\infty$ be a sequence of points whose $s$-coordinates $s_k\nearrow\infty$ as $k\nearrow\infty$. Let $U_k := (-k,\infty)\times S^{n}(r_0)$ be an exhaustion of the cylinder $\RR\times S^n(r_0)$. Then the translation map $s\mapsto (s+2k)$ defines an embedding $\psi_k: U_k\to\RR^N$, $V_k:=\psi_k(U_k)=(k,\infty)\times S^{n}(r_0)$. We claim that for $g_0=ds^2+\psi(s,\tau_0)^2 g_{\text{sph}}$,
\begin{align}\label{eq:CG}
\left. g_0\right\vert_{V_k} \overset{C^3}\longrightarrow g_{\text{cyl}}\text{ on compact subsets of } \RR\times S^n(r_0),
\end{align}
where $g_{\text{cyl}}= ds^2 + r_0^2 g_{\text{sph}}$ is the standard metric on the round cylinder. Without loss of generality, assume $r_0=1$. For all sufficiently large $k$, the $u$-coordinate of $y_k$ is bounded between $u_0$ and $1$. At $t=0$, recall that $\psi = C u$ for some constant $C=C(n)$, so $\p^i_s\psi \lesssim \p^i_s u$ ($i\in\NN$). Then at $\tau=\tau_0$, as $s_k\nearrow\infty$, $\psi\lesssim u \nearrow 1$, and hence from \eqref{eq:u_s}, we obtain
\begin{align*}
\psi_s & \lesssim u_s \lesssim \left(1-u^2\right)^{\frac{(\lambda+1)}{2}} \searrow 0,\\
\psi_{ss} & \lesssim u_{ss} \lesssim \left(1-u^2\right)^{\lambda}\searrow 0,\\
\psi_{sss} &\lesssim u_{sss} \lesssim \left(1-u^2\right)^{\frac{(3\lambda-1)}{2}}\searrow 0.
\end{align*}
This shows \eqref{eq:CG}\footnote{One checks that $\p^{i}_s\psi\lesssim \p^{i}_s u\lesssim \left(1-u^2\right)^{1+i\frac{(\lambda-1)}{2}}\searrow 0$ as $u\nearrow 1$ for all $i\in\NN$, so we in fact have convergence in smooth pointed Cheeger-Gromov topology.}, and hence $g_0$ satisfies (P3).

Therefore, $g_0\in\tilde{\CMcal{G}}^\ast_N$, and the set $\tilde{\CMcal{G}}^\ast_N$ is nonempty.

Since the sectional curvatures depend smoothly on the metric, there is an open set $\CMcal{G}^\ast_N \supset \tilde{\CMcal{G}}^\ast_N$ in $C^6$ topology such that $\CMcal{G}^\ast_N \subset \CMcal{G}_N$ and any $g\in\CMcal{G}^\ast_N$ satisfies P(1)--P(3). The lemma is now proved. \end{proof}

We now prove Theorem \ref{thm2}.
\begin{proof}[Proof of Theorem \ref{thm2}] By Lemma \ref{stansol}, the Ricci flow solution $g(t)$ on $\RR^N$ starting at any $g_0\in\CMcal{G}^\ast_N$ is a standard solution. Since $g_0\in\CMcal{G}_N$, Theorem \ref{thm1} applies to $g(t)$, and so Theorem \ref{thm2} follows.
\end{proof}



\begin{bibdiv}
\begin{biblist}

\bib{AV95}{article}{
   author={Angenent, S. B.},
   author={Vel{\'a}zquez, J. J. L.},
   title={Asymptotic shape of cusp singularities in curve shortening},
   journal={Duke Math. J.},
   volume={77},
   date={1995},
   number={1},
   pages={71\ndash 110},
}

\bib{AV97}{article}{
   author={Angenent, S. B.},
   author={Vel{\'a}zquez, J. J. L.},
   title={Degenerate neckpinches in mean curvature flow},
   journal={J. Reine Angew. Math.},
   volume={482},
   date={1997},
   pages={15\ndash 66},
}

\bib{AK04}{article}{
   author={Angenent, Sigurd},
   author={Knopf, Dan},
   title={An example of neckpinching for Ricci flow on $S^{n+1}$},
   journal={Math. Res. Lett.},
   volume={11},
   date={2004},
   number={4},
   pages={493--518},
}

\bib{ACK12}{article}{
   author={Angenent, Sigurd B.},
   author={Caputo, M. Cristina},
   author={Knopf, Dan},
   title={Minimally invasive surgery for Ricci flow singularities},
   journal={J. Reine Angew. Math.},
   volume={672},
   date={2012},
   pages={39--87},
}

\bib{AIK11}{article}{
   author={Angenent, Sigurd B.},
   author={Isenberg, James},
   author={Knopf, Dan},
   title={Formal matched asymptotics for degenerate Ricci flow neckpinches},
   journal={Nonlinearity},
   volume={24},
   date={2011},
   number={8},
   pages={2265--2280},
}


\bib{AIK12}{article}{
   author={Angenent, Sigurd B.},
   author={Isenberg, James},
   author={Knopf, Dan},
   title={Degenerate neckpinches in Ricci flow},
   journal={to appear in J. Reine Angew. Math.},
   eprint={arXiv:1208.4312},
}

\bib{AK07}{article}{
   author={Angenent, Sigurd B.},
   author={Knopf, Dan},
   title={Precise asymptotics of the Ricci flow neckpinch},
   journal={Comm. Anal. Geom.},
   volume={15},
   date={2007},
   number={4},
   pages={773--844},
}

\bib{BW08}{article}{
   author={B{\"o}hm, Christoph},
   author={Wilking, Burkhard},
   title={Manifolds with positive curvature operators are space forms},
   journal={Ann. of Math. (2)},
   volume={167},
   date={2008},
   number={3},
   pages={1079--1097},
}

\bib{B11}{article}{
   author={Brendle, Simon},
   title={Uniqueness of gradient Ricci solitons},
   journal={Math. Res. Lett.},
   volume={18},
   date={2011},
   number={3},
   pages={531--538},
}

\bib{B08}{article}{
   author={Brendle, Simon},
   title={A general convergence result for the Ricci flow in higher
   dimensions},
   journal={Duke Math. J.},
   volume={145},
   date={2008},
   number={3},
   pages={585--601},
}

\bib{B12-2}{article}{
   author={Brendle, Simon},
   title={Rotational symmetry of Ricci solitons in higher dimensions},
   journal={to appear in J. Differential Geom.},
   eprint={arXiv:1203.0270},
}

\bib{B13}{article}{
   author={Brendle, Simon},
   title={Rotational symmetry of self-similar solutions to the Ricci flow},
   journal={Invent. Math.},
   volume={194},
   date={2013},
   number={3},
   pages={731--764},
}

\bib{Bryant}{article}{
    author = {Bryant, R. L.},
     title = {Ricci flow solitons in dimension three with SO(3)-symmetries}, 
	journal = {available at \url{www.math.duke.edu/~bryant/3DRotSymRicciSolitons.pdf}},
}

\bib{CCh12}{article}{
   author={Cao, Huai-Dong},
   author={Chen, Qiang},
   title={On locally conformally flat gradient steady Ricci solitons},
   journal={Trans. Amer. Math. Soc.},
   volume={364},
   date={2012},
   number={5},
   pages={2377--2391},
}

\bib{ChZh06}{article}{
   author={Chen, Bing-Long},
   author={Zhu, Xi-Ping},
   title={Uniqueness of the Ricci flow on complete noncompact manifolds},
   journal={J. Differential Geom.},
   volume={74},
   date={2006},
   number={1},
   pages={119--154},
}

\bib{Ch91}{article}{
   author={Chen, Haiwen},
   title={Pointwise $\frac14$-pinched $4$-manifolds},
   journal={Ann. Global Anal. Geom.},
   volume={9},
   date={1991},
   number={2},
   pages={161--176},
}

\bib{ChW11}{article}{
   author = {Chen, Xiuxiong and Wang, Yuanqi},
    title = {On four-dimensional anti-self-dual gradient {R}icci solitons},
     date = {2011},	
  journal = {Preprint},
   eprint = {arXiv:1102.0358},
}

\bib{BCh91}{article}{
   author={Chow, Bennett},
   title={The Ricci flow on the $2$-sphere},
   journal={J. Differential Geom.},
   volume={33},
   date={1991},
   number={2},
   pages={325--334},
}

\bib{ChLN06}{book}{
   author={Chow, Bennett},
   author={Lu, Peng},
   author={Ni, Lei},
   title={Hamilton's Ricci flow},
   series={Graduate Studies in Mathematics},
   volume={77},
   publisher={American Mathematical Society},
   place={Providence, RI},
   date={2006},
   pages={xxxvi+608},
   isbn={978-0-8218-4231-7},
   isbn={0-8218-4231-5},
}

\bib{DdP07}{article}{
   author={Daskalopoulos, P.},
   author={del Pino, Manuel},
   title={Type II collapsing of maximal solutions to the Ricci flow in $\mathbb
   R^2$},
   journal={Ann. Inst. H. Poincar\'e Anal. Non Lin\'eaire},
   volume={24},
   date={2007},
   number={6},
   pages={851--874},
}

\bib{DH04}{article}{
   author={Daskalopoulos, P.},
   author={Hamilton, R.},
   title={Geometric estimates for the logarithmic fast diffusion equation},
   journal={Comm. Anal. Geom.},
   volume={12},
   date={2004},
   number={1-2},
   pages={143--164},
}

\bib{DS10}{article}{
   author={Daskalopoulos, Panagiota},
   author={Sesum, Natasa},
   title={Type II extinction profile of maximal solutions to the Ricci flow
   in $\mathbb R^2$},
   journal={J. Geom. Anal.},
   volume={20},
   date={2010},
   number={3},
   pages={565--591},
}

\bib{FIK03}{article}{
   author={Feldman, Mikhail},
   author={Ilmanen, Tom},
   author={Knopf, Dan},
   title={Rotationally symmetric shrinking and expanding gradient
   K\"ahler-Ricci solitons},
   journal={J. Differential Geom.},
   volume={65},
   date={2003},
   number={2},
   pages={169--209},
}

\bib{GI05}{article}{
   author={Garfinkle, David},
   author={Isenberg, James},
   title={Numerical studies of the behavior of Ricci flow},
   conference={
      title={Geometric evolution equations},
   },
   book={
      series={Contemp. Math.},
      volume={367},
      publisher={Amer. Math. Soc.},
      place={Providence, RI},
   },
   date={2005},
   pages={103--114},
}

\bib{GI08}{article}{
   author={Garfinkle, David},
   author={Isenberg, James},
   title={The modeling of degenerate neck pinch singularities in Ricci flow
   by Bryant solitons},
   journal={J. Math. Phys.},
   volume={49},
   date={2008},
   number={7},
   pages={073505, 10},
}

\bib{GZh08}{article}{
   author={Gu, Hui-Ling},
   author={Zhu, Xi-Ping},
   title={The existence of type II singularities for the Ricci flow on $S^{n+1}$},
   journal={Comm. Anal. Geom.},
   volume={16},
   date={2008},
   number={3},
   pages={467--494},
}

\bib{H82}{article}{
   author={Hamilton, Richard S.},
   title={Three-manifolds with positive Ricci curvature},
   journal={J. Differential Geom.},
   volume={17},
   date={1982},
   number={2},
   pages={255--306},
}

\bib{H86}{article}{
   author={Hamilton, Richard S.},
   title={Four-manifolds with positive curvature operator},
   journal={J. Differential Geom.},
   volume={24},
   date={1986},
   number={2},
   pages={153--179},
}

\bib{H88}{article}{
   author={Hamilton, Richard S.},
   title={The Ricci flow on surfaces},
   conference={
      title={Mathematics and general relativity},
      address={Santa Cruz, CA},
      date={1986},
   },
   book={
      series={Contemp. Math.},
      volume={71},
      publisher={Amer. Math. Soc.},
      place={Providence, RI},
   },
   date={1988},
   pages={237--262},
}

\bib{H95}{article}{
   author={Hamilton, Richard S.},
   title={The formation of singularities in the Ricci flow},
   conference={
      title={Surveys in differential geometry, Vol.\ II},
      address={Cambridge, MA},
      date={1993},
   },
   book={
      publisher={Int. Press, Cambridge, MA},
   },
   date={1995},
   pages={7--136},
}

\bib{Hui12}{article}{
   author={Hui, Kin Ming},
   title={Collasping behaviour of a singular diffusion equation},
   journal={Discrete Contin. Dyn. Syst.},
   volume={32},
   date={2012},
   number={6},
   pages={2165--2185},
}

\bib{I94}{article}{
   author={Ivey, Thomas},
   title={New examples of complete Ricci solitons},
   journal={Proc. Amer. Math. Soc.},
   volume={122},
   date={1994},
   number={1},
   pages={241--245},
}

\bib{K93}{article}{
    author = {King, John Robert},
     title = {Self-Similar behaviour for the equation of fast nonlinear diffusion},
   journal = {Phil. Trans. R. Soc., Lond. A},
    volume = {343},
    date = {1993},
    number = {1668},      
     pages = {337--375},
}

\bib{KL08}{article}{
   author={Kleiner, Bruce},
   author={Lott, John},
   title={Notes on Perelman's papers},
   journal={Geom. Topol.},
   volume={12},
   date={2008},
   number={5},
   pages={2587--2855},
}

\bib{Lu}{article}{
    author = {Lu, Peng},
     title = {Private communication},
}

\bib{LT}{article}{
    author = {Lu, Peng and Tian, Gang},
     title = {Uniqueness of standard solutions in the work of Perelman}, 
   journal ={available at \url{http://math.berkeley.edu/~lott/ricciflow/StanUniqWork2.pdf}},
}

\bib{M12}{article}{
    author = {M\'{a}ximo, Davi},
     title = {On the blow-up of four-dimensional Ricci flow singularities},
   journal = {to appear in J. Reine Angew. Math.},
    eprint = {arXiv:1204.5967},
}

\bib{P02}{article}{
    author = {Perelman, Grisha},
     title = {The entropy formula for the {R}icci flow and its geometric applications},
	journal = {Preprint},
	date = {2002},
	eprint = {arXiv:math/0211159},
}

\bib{P03-1}{article}{
   author = {Perelman, Grisha},
     title = {Ricci flow with surgery on three-manifolds}, 
	journal = {Preprint},
	date = {2003},
	eprint = {arXiv:math/0303109},
}

\bib{Shi89-1}{article}{
   author={Shi, Wan-Xiong},
   title={Deforming the metric on complete Riemannian manifolds},
   journal={J. Differential Geom.},
   volume={30},
   date={1989},
   number={1},
   pages={223--301},
}
	
\bib{S00}{article}{
   author={Simon, Miles},
   title={A class of Riemannian manifolds that pinch when evolved by Ricci
   flow},
   journal={Manuscripta Math.},
   volume={101},
   date={2000},
   number={1},
   pages={89--114},
}

\end{biblist}
\end{bibdiv}

\end{document}